\documentclass[11pt]{amsart} \textwidth=14.5cm \oddsidemargin=1cm
\usepackage{amsmath}
\usepackage{amsxtra}
\usepackage{amscd}
\usepackage{amsthm}
\usepackage{amsfonts}
\usepackage{amssymb}
\usepackage{eucal}
\usepackage{pmgraph}
\usepackage{stmaryrd}
\usepackage[usenames,dvipsnames]{color}

\input prepictex
\input pictex
\input postpictex

\newtheorem{thm}{Theorem}[section]
\newtheorem{cor}[thm]{Corollary}
\newtheorem{lem}[thm]{Lemma}
\newtheorem{prop}[thm]{Proposition}

\theoremstyle{definition}

\theoremstyle{remark}
\newtheorem{rem}[thm]{Remark}

\numberwithin{equation}{section}

\newtheoremstyle{dotless}{5pt}{3pt}{\itshape}{}{}{)}{ }{}
\theoremstyle{dotless}

\begin{document}

\newcommand{\thmref}[1]{Theorem~\ref{#1}}
\newcommand{\secref}[1]{Section~\ref{#1}}
\newcommand{\lemref}[1]{Lemma~\ref{#1}}
\newcommand{\propref}[1]{Proposition~\ref{#1}}
\newcommand{\corref}[1]{Corollary~\ref{#1}}
\newcommand{\remref}[1]{Remark~\ref{#1}}
\newcommand{\eqnref}[1]{(\ref{#1})}
\newcommand{\exref}[1]{Example~\ref{#1}}

\newcommand{\nc}{\newcommand}
\nc{\bi}{\bibitem}

\nc{\mf}{\mathfrak}
\nc{\mc}{\mathcal}
\nc{\ov}{\overline}
\nc{\wt}{\widetilde}

\nc{\Z}{{\mathbb Z}}
\nc{\C}{{\mathbb C}}
\nc{\N}{{\mathbb N}}
\nc{\F}{{\mf F}}
\nc{\la}{\lambda}
\nc{\ep}{\epsilon}
\nc{\La}{\Lambda}

\nc{\G}{{\mathfrak g}}
\nc{\DG}{\widetilde{\mathfrak g}}
\nc{\SG}{\overline{\mathfrak g}}
\nc{\Gd}{{\mathfrak g}^\diamond}
\nc{\SGd}{\overline{\mathfrak g}^\diamond}

\nc{\h}{{\mathfrak h}}
\nc{\Dh}{\widetilde{\mathfrak h}}
\nc{\Sh}{\overline{\mathfrak h}}
\nc{\hd}{{\mathfrak h}^\diamond}
\nc{\Shd}{\overline{\mathfrak h}^\diamond}

\nc{\cl}{{\mathfrak l}}
\nc{\Dl}{\widetilde{\mathfrak l}}
\nc{\Sl}{\overline{\mathfrak l}}
\nc{\ld}{{\mathfrak l}^\diamond}
\nc{\Sld}{\overline{\mathfrak l}^\diamond}

\nc{\clf}{{\mathfrak l}(m,n)}
\nc{\Dlf}{\widetilde{\mathfrak l}(m,n)}
\nc{\Slf}{\overline{\mathfrak l}(m,n)}
\nc{\ldf}{{\mathfrak l}^\diamond(m,n)}
\nc{\Sldf}{\overline{\mathfrak l}^\diamond(m,n)}

\nc{\n}{\mf n}

\nc{\D}{\mc D} \nc{\Li}{{\mc L}}
\nc{\is}{{\mathbf i}} \nc{\V}{\mf V}  \nc{\NS}{\mf N}
\nc{\dt}{\mathord{\hbox{${\frac{d}{d t}}$}}} \nc{\E}{\mc E}
\nc{\ba}{\tilde{\pa}} \nc{\half}{\frac{1}{2}}
 \nc{\hf}{\frac{1}{2}}
\nc{\hgl}{\widehat{\mathfrak{gl}}} \nc{\gl}{{\mathfrak{gl}}}
\nc{\hz}{\hf+\Z}
\nc{\dinfty}{{\infty\vert\infty}} \nc{\SLa}{\overline{\Lambda}}
\nc{\SF}{\overline{\mathfrak F}} \nc{\SP}{\overline{\mathcal P}}
\nc{\U}{\mathfrak u} \nc{\SU}{\overline{\mathfrak u}}

\nc{\osp}{\mf{osp}}
\nc{\spo}{\mf{spo}}
\nc{\hosp}{\widehat{\mf{osp}}}
\nc{\hspo}{\widehat{\mf{spo}}}
\nc{\I}{\mathbb{I}}
\nc{\X}{\mathbb{X}}
\nc{\Y}{\mathbb{Y}}
\nc{\hh}{\widehat{\mf{h}}}
\nc{\cc}{{\mathfrak c}}
\nc{\dd}{{\mathfrak d}}
\nc{\aaa}{{\mf A}}
\nc{\xx}{{\mf x}}
\nc{\wty}{\widetilde{\mathbb Y}}
\nc{\ovy}{\overline{\mathbb Y}}
\nc{\vep}{\bar{\epsilon}}

\nc{\On}{{\mc O}}
\nc{\Otn}{\wt{\mc O}}
\nc{\Orn}{\ov{\mc O}}
\nc{\Oln}{\mc{O}^{\diamond}}
\nc{\Olrn}{\ov{\mc O}^{\diamond}}

\nc{\Otnb}{\wt{\mc O}^{\bar{0}}}
\nc{\Ornb}{\ov{\mc O}^{\bar{0}}}
\nc{\Olnb}{\mc{O}^{\diamond {\bar{0}}}}
\nc{\Olrnb}{\ov{\mc O}^{\diamond {\bar{0}}}}

\nc{\Oa}{{\mc O}_a}
\nc{\Ot}{\wt{\mc O}_a}
\nc{\Or}{\ov{\mc O}_a}
\nc{\Ol}{\mc{O}^{\diamond}_a}
\nc{\Olr}{\ov{\mc O}^{\diamond}_a}

\nc{\Otb}{\wt{\mc O}^{\bar{0}}_a}
\nc{\Orb}{\ov{\mc O}^{\bar{0}}_a}
\nc{\Olb}{\mc{O}^{\diamond {\bar{0}}}_a}
\nc{\Olrb}{\ov{\mc O}^{\diamond {\bar{0}}}_a}

\nc{\Oaf}{{\mc O}_a(m,n)}
\nc{\Otf}{\wt{\mc O}^{\bar{0}}_a(m,n)}
\nc{\Orf}{\ov{\mc O}^{\bar{0}}_a(m,n)}
\nc{\Olf}{\mc{O}^{\diamond {\bar{0}}}_a(m,n)}
\nc{\Olrf}{\ov{\mc O}^{\diamond {\bar{0}}}_a(m,n)}

\nc{\Oanf}{{\mc O}(m,n)}
\nc{\Otnf}{\wt{\mc O}^{\bar{0}}(m,n)}
\nc{\Ornf}{\ov{\mc O}^{\bar{0}}(m,n)}
\nc{\Olnf}{\mc{O}^{\diamond {\bar{0}}}(m,n)}
\nc{\Olrnf}{\ov{\mc O}^{\diamond {\bar{0}}}(m,n)}

\nc{\glinf}{\widehat{\gl}_\infty} \nc{\VV}{\mathbb V}
\nc{\WW}{\mathbb V^*}

\advance\headheight by 2pt

\title[Super duality and applications]{Super duality for general
linear Lie superalgebras and applications}

\author[Cheng]{Shun-Jen Cheng}
\address{Institute of Mathematics, Academia Sinica, Taipei,
Taiwan 11529} \email{chengsj@math.sinica.edu.tw}

\author[Lam]{Ngau Lam}
\address{Department of Mathematics, National Cheng-Kung University, Tainan, Taiwan 70101}
\email{nlam@mail.ncku.edu.tw}

\author[Wang]{Weiqiang Wang}
\address{Department of Mathematics, University of Virginia, Charlottesville, VA 22904}
\email{ww9c@virginia.edu}

\begin{abstract}
We apply the super duality formalism recently developed by the
authors to obtain new equivalences of various module categories of
general linear Lie superalgebras. We establish the correspondence of
standard, tilting, and simple modules, as well as the identification
of the $\mf u$-homology groups, under these category equivalences.
As an application, we obtain a complete solution of the irreducible
character problem for some new parabolic BGG categories of
$\gl(m|n)$-modules, including the full BGG category of
$\gl(m|2)$-modules, in terms of type $A$ Kazhdan-Lusztig
polynomials.
\end{abstract}



 \maketitle

\section{Introduction}

Super duality is a powerful new approach developed in the past few
years in the study of representation theory of Lie superalgebras. It
provides a surprising direct link between various parabolic BGG
categories of modules of Lie superalgebras and Lie algebras, and
allows us to solve the fundamental irreducible character problem for
Lie superalgebras in terms of Kazhdan-Lusztig polynomials for
semisimple (or more generally Kac-Moody) Lie algebras.  Super
duality was first formulated as conjectures for general linear Lie
superalgebras in \cite{CWZ, CW}, and then formulated and established
in much generality in \cite{CL, CLW} recently. Our super duality
solution of the irreducible character problem for Lie superalgebras
ultimately depends on the solution \cite{BB, BK} of Kazhdan-Lusztig
conjecture for Lie algebras \cite{KL, Deo}. We also refer to
\cite{BS} for a completely different approach towards the very
special version of super duality conjecture as formulated in \cite{CWZ}.

The goal of this paper is to formulate and establish a new form of
super duality (which is an equivalence of categories) for the general
linear Lie superalgebras. The super duality formalism typically
takes advantage of a type $A$ branch which many Dynkin diagrams
possess. In this paper we will apply the super duality formalism
simultaneously to the two ends of the (super type $A$) Dynkin
diagrams. We in addition establish a correspondence of tilting
modules under super duality, which was omitted in our previous
paper \cite{CLW}, besides the correspondences of standard and simple
modules under super duality. We also show that super duality
is actually an equivalence of {\em tensor} categories, which was not
addressed in our earlier work.

The version of super duality in this paper allows us to settle the
irreducible character problem in some new parabolic BGG categories
of $\gl(m|n)$-modules  not covered in \cite{CW, CL, CLW} in terms of
parabolic Kazhdan-Lusztig polynomials of type $A$. Brundan
\cite{Br1} conjectured that the full BGG category of
$\gl(m|n)$-modules (of integer weights) categorifies the Fock space
$\VV^{\otimes m} \otimes (\VV^*)^{\otimes n}$, and he established a
maximal parabolic version of this  conjecture for the category of
finite-dimensional $\gl(m|n)$-modules, where $\VV$ and $\VV^*$
denote the natural $U_q(\gl_\infty)$-module and its dual module,
respectively. Various {\em parabolic} versions of Brundan's
conjecture have been formulated and settled since then via the super
duality approach (\cite{ CW, CL}). However, the irreducible
character problem in the {\em full} BGG category for general $m$ and
$n$ remained unsettled in {\em loc. cit.}, except when $m$ or $n$ is
at most $1$. Our paper confirms a variant of Brundan's conjecture on
the {\em full} BGG category of $\gl(m|n)$-modules (with respect to a
nonstandard Borel subalgebra of block type $1|m|1$) for $n=2$ or
$m=2$, and we will return to address the case for general $m$ and
$n$ in another paper.

Recall the infinite-dimensional Lie algebra $\glinf$ (which is a
central extension with a central element denoted by $K$) has played
a fundamental role in many different contexts. If $K$ acts as a
scalar $\ell$ on a $\glinf$-module $M$, then $\ell$  is called the
level of the module $M$. Let $\ell$ be a positive integer. A very
special case of super duality states that the (semisimple)
category of integrable $\glinf$-modules of positive level $\ell$ is
equivalent to a suitable category of $\glinf$-modules of negative
level $-\ell$ (which is far from being obvious to be semisimple); A
weak version of the super duality on the Grothendieck group level
suffices to recover the character formulas established by completely
different methods in \cite{KR} and \cite[Remark 5.2]{CL0}.

Following Vogan \cite{V}, a computation of the $\mf u$-homology groups
is basically a computation of the (parabolic) Kazhdan-Lusztig
polynomials. Super duality identifies the corresponding $\mf
u$-homology groups with coefficients in modules belonging to
different categories. This in particular allows us to recover easily
the computation of $\mf u$-homology groups with coefficients in
modules appearing in Howe duality decompositions (which were
computed by completely different techniques in \cite{CK, CKW, HLT,
LZ}).

To keep the paper at a reasonable length, we have omitted several
proofs leading towards super duality, when they are similar to the ones
in \cite{CL, CLW}. We also refer directly to
\cite{Br1, CW} for the Fock space formulation of some parabolic
versions of Brundan's conjecture.

We shall use the following notations throughout this paper. The
symbols $\Z$, $\N$, and $\Z_+$ stand for the sets of all, positive,
and non-negative integers, respectively. For a superspace
$V=V_{\bar{0}}\oplus V_{\bar{1}}$ and a homogeneous element $v\in
V$, we use the notation $|v|$ to denote the $\Z_2$-degree of $v$.
Finally all vector spaces, algebras, tensor products, et cetera, are
over the field of complex numbers $\C$.

\section{Lie superalgebras associated to various Dynkin diagrams} \label{sec:superalgebras}

In this section, we introduce a general linear Lie superalgebra
$\DG$ of infinite rank and its subalgebras $\G$, $\SG$,
$\G^\diamond$ and $\SG^\diamond$. We also formulate the finite-rank
counterparts of these Lie superalgebras.

\subsection{General linear Lie superalgebra}  \label{sec:glhat}

We fix $k\in\Z_+$. We consider the following ordered set
$\wt{\I}$:
\begin{align}\label{order:standard}
\cdots <-{\frac{3}{2}}<-{1}<-{\hf}<0<\underbrace{\ov{1}<\ov{2}
<\cdots<\ov{k}}_{k}<\hf<1<\frac{3}{2}<\cdots,
\end{align}
For $k>0$, set ${\mathbb K}=\{\ov{1},\ldots,\ov{k}\}$, and for
$k=0$, set ${\mathbb K}=\emptyset$. We define the following subsets
of $\wt{\mathbb I}$:
\begin{align*}
&{\I} :={\mathbb K}\cup \Z,
&\ov{\I} :={\mathbb K}\cup -\Z_+\cup (\hf+\Z_+),\\
&{\I}^\diamond :={\mathbb K}\cup (-\hf-\Z_+)\cup \N,
&\ov{\I}^\diamond :={\mathbb K}\cup (\hf+\Z).
\end{align*}

Consider the infinite-dimensional superspace $\wt{V}$ with ordered
basis $\{v_i|i\in\wt{\I}\}$. We declare $|v_r|=\bar{0}$, if
$r\in\Z\cup\mathbb{K}$, and $|v_r|=\bar{1}$, if $r\in\hf+\Z$.  With
respect to this basis, a linear map on $\wt{V}$ may be identified
with a complex matrix $(a_{rs})_{r,s\in\wt{\I}}$.  The Lie
superalgebra $\gl(\wt{V})$ is the Lie subalgebra of linear
transformations on $\wt{V}$ consisting of $(a_{rs})$ with $a_{rs}=0$
for all but finitely many $a_{rs}$'s. Denote by
$E_{rs}\in\gl(\wt{V})$ the elementary matrix with $1$ at the $r$th
row and $s$th column and zero elsewhere. Denote by
$\DG:=\gl(\wt{V})\oplus\C K$ the central extension of $\gl(\wt{V})$
by a one-dimensional center $\C K$ determined by the $2$-cocycle
\begin{align*}
\alpha(A,B):=\text{Str}([\mathfrak{J},A]B),\quad A,B\in\gl(\wt{V}),
\end{align*}
where $\mf{J}=\sum_{r\le 0}E_{rr}$ and $\text{Str}$ denotes the
supertrace. Observe that the cocycle $\alpha$ is a coboundary.
Indeed, there is embedding $\iota:\gl(\wt{V})\rightarrow \DG$,
defined by sending $A\in \gl(\wt{V})$ to $A+{\rm Str}(\mf{J}A)K$
(cf. \cite[Section 2.5]{CLW}). It is clear that $\iota(\gl(\wt{V}))$
is an ideal of $\DG$ and  $\DG$ is a direct sum of the ideals
$\iota(\gl(\wt{V}))$ and $\C K$. For $A\in\gl(\wt{V})$ we denote
\begin{align*}
\widehat{A}:=\iota(A)-\text{Str}(\mf{J}A)K\in\DG.
\end{align*}

Let $\wt{\h}$ and $\wt{\mf{b}}$ denote the standard Cartan
subalgebra $\oplus_{r\in\wt{\I}}\C \widehat{E}_{rr}\oplus\C K$ and
the standard Borel subalgebra $\oplus_{r\le s,r,s\in\wt{\I}}\C
\widehat{E}_{rs}\oplus\C K$, respectively. Define $\La_0\in
\wt{\h}^*$ by
$$
\La_0(K)=1, \qquad  \La_0(\widehat{E}_{rr})=0,\; \forall
r\in\wt{\I}.
$$
Let $\epsilon_i\in \wt{\h}^*$ be determined by $\langle
\epsilon_i,\widehat{E}_{jj}\rangle=\delta_{ij}$ for $i,j\in\wt{\I}$ and
$\langle\epsilon_i,K\rangle=0$. We set
\begin{align*}
\alpha_{r}:=\epsilon_{r}-\epsilon_{r+\hf},\quad   r\in\hf\Z,
\qquad{\rm and }\quad\alpha_{\ov{j}}
:=\epsilon_{\ov{j}}-\epsilon_{\ov{j+1}},\quad 1\le j\le k-1.
\end{align*}
The totally ordered set $\wt{\I}$ gives rise to a positive system for
$\wt\G$
$$
\wt{\Delta}^+ =\{\,\epsilon_i-\epsilon_j\,\mid\, i<j,\,
i,j\in\wt{\I}\,\},
$$
whose associated fundamental system is given by
\begin{align*}
\wt{\Pi}&=
\begin{cases}
\{\,\alpha_r\,\mid\, r\in \hf\Z\setminus\{0\}\,\}
\cup\{\epsilon_0-\epsilon_{\ov{1}},\,\epsilon_{\ov{k}}-\epsilon_{\hf},\,
\alpha_{\ov{1}},\,\alpha_{\ov{2}},\cdots\,\alpha_{\ov{k-1}}\,\},\quad &\text{if $k\not=0$};\\
\{\,\alpha_r\,\mid\, r\in \hf\Z\,\}, \quad &\text{if $k=0$}.
\end{cases}
\end{align*}
Let \makebox(23,0){$\oval(20,14)$}\makebox(-20,8){${\mc{K}}$} denote
the following Dynkin diagram with fundamental system:
\begin{center}
\hskip -3cm \setlength{\unitlength}{0.16in}
\begin{picture}(24,3)
\put(8,2){\makebox(0,0)[c]{$\bigcirc$}}
\put(10.4,2){\makebox(0,0)[c]{$\bigcirc$}}
\put(14.85,2){\makebox(0,0)[c]{$\bigcirc$}}
\put(17.25,2){\makebox(0,0)[c]{$\bigcirc$}}
\put(5.6,2){\makebox(0,0)[c]{$\bigcirc$}}
\put(8.4,2){\line(1,0){1.55}} \put(10.82,2){\line(1,0){0.8}}
\put(13.2,2){\line(1,0){1.2}} \put(15.28,2){\line(1,0){1.45}}
\put(6,2){\line(1,0){1.4}}
\put(12.5,1.95){\makebox(0,0)[c]{$\cdots$}}
\put(5.5,1){\makebox(0,0)[c]{\tiny$\alpha_{\ov{1}}$}}
\put(8,1){\makebox(0,0)[c]{\tiny$\alpha_{\ov{2}}$}}
\put(10.3,1){\makebox(0,0)[c]{\tiny$\alpha_{\ov{3}}$}}
\put(15,1){\makebox(0,0)[c]{\tiny$\alpha_{\ov{k-2}}$}}
\put(17.2,1){\makebox(0,0)[c]{\tiny$\alpha_{\ov{k-1}}$}}
\end{picture}
\end{center}
The following is a Dynkin diagram together with the fundamental
system $\wt\Pi$ for $\wt{\G}$, where $\bigotimes$ denotes an odd
isotropic simple root:
\begin{center}\label{Dynkin:gl}
\hskip -1cm \setlength{\unitlength}{0.16in}
\begin{picture}(24,3)
\put(-3,2){\makebox(0,0)[c]{$D(\DG)$:}}
\put(8,2){\makebox(0,0)[c]{$\bigotimes$}}
\put(10.4,2){\makebox(0,0)[c]{$\bigcirc$}}
\put(16,2){\makebox(0,0)[c]{$\bigotimes$}}
\put(18.35,2){\makebox(0,0)[c]{$\bigotimes$}}
\put(20.8,2){\makebox(0,0)[c]{$\bigotimes$}}
\put(5.6,2){\makebox(0,0)[c]{$\bigotimes$}}
\put(8.4,2){\line(1,0){1.6}} \put(10.8,2){\line(1,0){1.5}}
\put(14,2){\line(1,0){1.5}} \put(16.4,2){\line(1,0){1.5}}
\put(18.8,2){\line(1,0){1.5}} \put(21.2,2){\line(1,0){1.5}}
\put(6,2){\line(1,0){1.5}}
\put(3.65,2){\line(1,0){1.5}}
\put(2.5,1.95){\makebox(0,0)[c]{$\cdots$}}
\put(13.15,1.95){\makebox(0,0)[c]{{\ovalBox(1.6,1.2){$\mc{K}$}}}}
\put(23.9,1.95){\makebox(0,0)[c]{$\cdots$}}
\put(8,1){\makebox(0,0)[c]{\tiny $\alpha_{-\hf}$}}
\put(10.5,1){\makebox(0,0)[c]{\tiny $\epsilon_0-\epsilon_{\ov{1}}$}}
\put(15.8,1){\makebox(0,0)[c]{\tiny $\epsilon_{\ov{k}}-\epsilon_{\hf}$}}
\put(18.4,1){\makebox(0,0)[c]{\tiny $\alpha_{\hf}$}}
\put(21,1){\makebox(0,0)[c]{\tiny $\alpha_{1}$}}
\put(5.4,1){\makebox(0,0)[c]{\tiny $\alpha_{-1}$}}
\end{picture}
\end{center}
When $k=0$, the diagram above means that the middle three terms are
replaced by the odd isotropic simple root $\epsilon_0-\epsilon_\hf$.
The Dynkin diagrams in the rest of the paper are interpreted in the
similar way for $k=0$.

Define
\begin{equation*}
\vert r\vert =
\begin{cases}0\quad&{\rm if}\,\,r\in \I;\\
1\quad&{\rm if}\,\, r\in \wt{\I}\backslash\I.
\end{cases}\\
\end{equation*}
Let $\tau$ be an automorphism of $\DG$ of order $4$ defined by
\begin{equation}\label{tau}
\tau(\widehat{E}_{rs})
:=-(-1)^{\vert r\vert(\vert r\vert+\vert s\vert)}\widehat{E}_{sr},\quad\tau(K)=-K.
\end{equation}

\subsection{The subalgebras of $\wt{\G}$}\label{subalgebras}

The subalgebra of $\wt{\G}$ generated by $K$ and $\widehat{E}_{rs}$
with $r,s\in {\I}$ (respectively $\ov{\I}$, ${\I}^\diamond$,
$\ov{\I}^\diamond$) is denoted by $\G$ (respectively $\SG$,
$\G^\diamond$ and $\SG^\diamond$). The Cartan subalgebras
(respectively Borel subalgebras) of $\G$, $\SG$, $\G^\diamond$ and
$\SG^\diamond$ induced from $\DG$ are denoted by $\mf{h}$,
$\ov{\mf{h}}$, $\mf{h}^\diamond$ and $\ov{\mf{h}}^\diamond$
(respectively $\mf{b}$, $\ov{\mf{b}}$, $\mf{b}^\diamond$ and
$\ov{\mf{b}}^\diamond$) respectively. We set
$$
\beta_{r}:=\epsilon_{r}-\epsilon_{r+1}, \quad \forall r\in\hf\Z.
$$
For $n\in\Z_+$, let
\makebox(23,0){$\oval(20,14)$}\makebox(-20,8){$\mc{L}_n$},
\makebox(23,0){$\oval(20,14)$}\makebox(-20,8){$\mc{L}_n^\diamond$},
\makebox(23,0){$\oval(20,14)$}\makebox(-20,8){${\mc{R}}_n$} and
\makebox(23,0){$\oval(20,14)$}\makebox(-20,8){$\ov{\mc{R}}_n$}
denote the following Dynkin diagrams with fixed fundamental systems:

\begin{center}
\hskip -3cm \setlength{\unitlength}{0.16in}
\begin{picture}(24,3)
\put(8,2){\makebox(0,0)[c]{$\bigcirc$}}
\put(10.4,2){\makebox(0,0)[c]{$\bigcirc$}}
\put(14.85,2){\makebox(0,0)[c]{$\bigcirc$}}
\put(17.25,2){\makebox(0,0)[c]{$\bigcirc$}}
\put(5.6,2){\makebox(0,0)[c]{$\bigcirc$}}
\put(8.4,2){\line(1,0){1.55}} \put(10.82,2){\line(1,0){0.8}}
\put(13.2,2){\line(1,0){1.2}} \put(15.28,2){\line(1,0){1.45}}
\put(6,2){\line(1,0){1.4}}
\put(12.5,1.95){\makebox(0,0)[c]{$\cdots$}}
\put(0,1.2){{\ovalBox(1.6,1.2){$\mc{L}_n$}}}
\put(5.5,1){\makebox(0,0)[c]{\tiny$\beta_{1-n}$}}
\put(8,1){\makebox(0,0)[c]{\tiny$\beta_{2-n}$}}
\put(10.3,1){\makebox(0,0)[c]{\tiny$\beta_{3-n}$}}
\put(15,1){\makebox(0,0)[c]{\tiny$\beta_{-2}$}}
\put(17.2,1){\makebox(0,0)[c]{\tiny$\beta_{-1}$}}
\end{picture}
\end{center}
\begin{center}
\hskip -3cm \setlength{\unitlength}{0.16in}
\begin{picture}(24,3)
\put(8,2){\makebox(0,0)[c]{$\bigcirc$}}
\put(10.4,2){\makebox(0,0)[c]{$\bigcirc$}}
\put(14.85,2){\makebox(0,0)[c]{$\bigcirc$}}
\put(17.25,2){\makebox(0,0)[c]{$\bigcirc$}}
\put(5.6,2){\makebox(0,0)[c]{$\bigcirc$}}
\put(8.4,2){\line(1,0){1.55}} \put(10.82,2){\line(1,0){0.8}}
\put(13.2,2){\line(1,0){1.2}} \put(15.28,2){\line(1,0){1.45}}
\put(6,2){\line(1,0){1.4}}
\put(12.5,1.95){\makebox(0,0)[c]{$\cdots$}}
\put(0,1.2){{\ovalBox(1.6,1.2){$\mc{L}_n^\diamond$}}}
\put(5.5,1){\makebox(0,0)[c]{\tiny$\beta_{\hf-n}$}}
\put(8,1){\makebox(0,0)[c]{\tiny$\beta_{\frac{3}{2}-n}$}}
\put(10.3,1){\makebox(0,0)[c]{\tiny$\beta_{\frac{5}{2}-n}$}}
\put(15,1){\makebox(0,0)[c]{\tiny$\beta_{-\frac{5}{2}}$}}
\put(17.2,1){\makebox(0,0)[c]{\tiny$\beta_{-\frac{3}{2}}$}}
\end{picture}
\end{center}
\begin{center}
\hskip -3cm \setlength{\unitlength}{0.16in}
\begin{picture}(24,3)
\put(8,2){\makebox(0,0)[c]{$\bigcirc$}}
\put(10.4,2){\makebox(0,0)[c]{$\bigcirc$}}
\put(14.85,2){\makebox(0,0)[c]{$\bigcirc$}}
\put(17.25,2){\makebox(0,0)[c]{$\bigcirc$}}
\put(5.6,2){\makebox(0,0)[c]{$\bigcirc$}}
\put(8.4,2){\line(1,0){1.55}} \put(10.82,2){\line(1,0){0.8}}
\put(13.2,2){\line(1,0){1.2}} \put(15.28,2){\line(1,0){1.45}}
\put(6,2){\line(1,0){1.4}}
\put(12.5,1.95){\makebox(0,0)[c]{$\cdots$}}
\put(0,1.2){{\ovalBox(1.6,1.2){$\mc{R}_n$}}}
\put(5.5,1){\makebox(0,0)[c]{\tiny$\beta_{1}$}}
\put(8,1){\makebox(0,0)[c]{\tiny$\beta_{2}$}}
\put(10.3,1){\makebox(0,0)[c]{\tiny$\beta_{3}$}}
\put(15,1){\makebox(0,0)[c]{\tiny$\beta_{n-2}$}}
\put(17.2,1){\makebox(0,0)[c]{\tiny$\beta_{n-1}$}}
\end{picture}
\end{center}
\begin{center}
\hskip -3cm \setlength{\unitlength}{0.16in}
\begin{picture}(24,3)
\put(8,2){\makebox(0,0)[c]{$\bigcirc$}}
\put(10.4,2){\makebox(0,0)[c]{$\bigcirc$}}
\put(14.85,2){\makebox(0,0)[c]{$\bigcirc$}}
\put(17.25,2){\makebox(0,0)[c]{$\bigcirc$}}
\put(5.6,2){\makebox(0,0)[c]{$\bigcirc$}}
\put(8.4,2){\line(1,0){1.55}} \put(10.82,2){\line(1,0){0.8}}
\put(13.2,2){\line(1,0){1.2}} \put(15.28,2){\line(1,0){1.45}}
\put(6,2){\line(1,0){1.4}}
\put(12.5,1.95){\makebox(0,0)[c]{$\cdots$}}
\put(0,1.2){{\ovalBox(1.6,1.2){$\ov{\mc{R}}_n$}}}
\put(5.5,1){\makebox(0,0)[c]{\tiny$\beta_{\hf}$}}
\put(8,1){\makebox(0,0)[c]{\tiny$\beta_{\frac{3}{2}}$}}
\put(10.3,1){\makebox(0,0)[c]{\tiny$\beta_{\frac{5}{2}}$}}
\put(15,1){\makebox(0,0)[c]{\tiny$\beta_{n-\frac{5}{2}}$}}
\put(17.2,1){\makebox(0,0)[c]{\tiny$\beta_{n-\frac{3}{2}}$}}
\end{picture}
\end{center}
In the limit $n\mapsto\infty$, the associated Dynkin diagrams of
\makebox(23,0){$\oval(20,14)$}\makebox(-20,8){$\mc{L}_n$},
\makebox(23,0){$\oval(20,14)$}\makebox(-20,8){$\mc{L}_n^\diamond$},
\makebox(23,0){$\oval(20,14)$}\makebox(-20,8){${\mc{R}}_n$} and
\makebox(23,0){$\oval(20,14)$}\makebox(-20,8){$\ov{\mc{R}}_n$} are
denoted by \makebox(23,0){$\oval(20,14)$}\makebox(-20,8){$\mc{L}$},
\makebox(23,0){$\oval(20,14)$}\makebox(-20,8){$\mc{L}^\diamond$},
\makebox(23,0){$\oval(20,14)$}\makebox(-20,8){${\mc{R}}$} and
\makebox(23,0){$\oval(20,14)$}\makebox(-20,8){$\ov{\mc{R}}$},
respectively.

For $m,n\in \Z_+\cup\{\infty\}$, let $\DG(m,n)$ denote the
subalgebra of $\DG$ generated by $E_{rs}$ with $-m< r,s< n+1$ and
$r,s\in \wt{\I}$. Note that $\DG(m,n)$ is a finite-dimensional
subalgebra for $m,n\in \Z_+$. Set
\begin{eqnarray*}
\G(m,n)=\G\cap\DG(m,n), &&\quad \SG(m,n)=\SG\cap\DG(m,n),
  \\
\G^\diamond(m,n)=\G^\diamond\cap\DG(m,n), && \quad
\SG^\diamond(m,n)=\SG^\diamond\cap\DG(m,n).
\end{eqnarray*}
Then $\wt{\mf{h}}(m,n)=\wt{\mf{h}}\cap \DG(m,n)$ and
$\wt{\mf{b}}(m,n)=\wt{\mf{b}}\cap \DG(m,n)$ are the Cartan and Borel
subalgebras of $\DG(m,n)$, and we have four variants of Cartan and
Borel subalgebras for the remaining four Lie superalgebras with self-explanatory
notations. The Dynkin diagrams with given fundamental systems of
$\G(m,n)$, $\SG(m,n)$, $\G^\diamond(m,n)$ and $\SG^\diamond(m,n)$
with respect to the Borel subalgebras $\mf{b}(m,n)$,
$\ov{\mf{b}}(m,n)$, $\mf{b}^\diamond(m,n)$ and
$\ov{\mf{b}}^\diamond(m,n)$, are given as follows:
\bigskip
\begin{center}
\hskip -3cm \setlength{\unitlength}{0.16in}
\begin{picture}(24,1)
\put(3,0.5){\makebox(0,0)[c]{${\mathcal D}(\G(m,n))$:}}
\put(9.25,0.5){\makebox(0,0)[c]{{\ovalBox(1.6,1.2){$\mc{L}_m$}}}}
\put(10.05,0.5){\line(1,0){1.85}}
\put(12.35,0.5){\makebox(0,0)[c]{$\bigcirc$}}
\put(12.8,0.5){\line(1,0){1.85}}
\put(15.45,0.5){\makebox(0,0)[c]{{\ovalBox(1.6,1.2){$\mc{K}$}}}}
\put(16.25,0.5){\line(1,0){1.85}}
\put(18.55,0.5){\makebox(0,0)[c]{$\bigcirc$}}
\put(19.0,0.5){\line(1,0){1.85}}
\put(21.65,0.5){\makebox(0,0)[c]{{\ovalBox(1.6,1.2){$\mc{R}_n$}}}}
\put(12.5,-0.5){\makebox(0,0)[c]{\tiny$\epsilon_{0}-\epsilon_{\bar{1}}$}}
\put(18.5,-0.5){\makebox(0,0)[c]{\tiny$\epsilon_{\ov{k}}-\epsilon_{1}$}}
\end{picture}
\end{center}
\bigskip
\begin{center}
\hskip -3cm \setlength{\unitlength}{0.16in}
\begin{picture}(24,1)
\put(3,0.5){\makebox(0,0)[c]{${\mathcal D}(\SG(m,n))$:}}
\put(9.25,0.5){\makebox(0,0)[c]{{\ovalBox(1.6,1.2){$\mc{L}_m$}}}}
\put(10.05,0.5){\line(1,0){1.85}}
\put(12.35,0.5){\makebox(0,0)[c]{$\bigcirc$}}
\put(12.8,0.5){\line(1,0){1.85}}
\put(15.45,0.5){\makebox(0,0)[c]{{\ovalBox(1.6,1.2){$\mc{K}$}}}}
\put(16.25,0.5){\line(1,0){1.85}}
\put(18.55,0.5){\makebox(0,0)[c]{$\bigotimes$}}
\put(19.0,0.5){\line(1,0){1.85}}
\put(21.65,0.5){\makebox(0,0)[c]{{\ovalBox(1.6,1.2){$\ov{\mc{R}}_n$}}}}
\put(12.5,-0.5){\makebox(0,0)[c]{\tiny$\epsilon_0-\epsilon_{\bar{1}}$}}
\put(18.5,-0.5){\makebox(0,0)[c]{\tiny$\epsilon_{\ov{k}}-\epsilon_{\hf}$}}
\end{picture}
\end{center}
\bigskip
\begin{center}
\hskip -3cm \setlength{\unitlength}{0.16in}
\begin{picture}(24,1)
\put(3,0.5){\makebox(0,0)[c]{${\mathcal D}(\G^\diamond(m,n))$:}}
\put(9.25,0.5){\makebox(0,0)[c]{{\ovalBox(1.6,1.2){$\mc{L}^\diamond_m$}}}}
\put(10.05,0.5){\line(1,0){1.85}}
\put(12.35,0.5){\makebox(0,0)[c]{$\bigotimes$}}
\put(12.8,0.5){\line(1,0){1.85}}
\put(15.45,0.5){\makebox(0,0)[c]{{\ovalBox(1.6,1.2){$\mc{K}$}}}}
\put(16.25,0.5){\line(1,0){1.85}}
\put(18.55,0.5){\makebox(0,0)[c]{$\bigcirc$}}
\put(19.0,0.5){\line(1,0){1.85}}
\put(21.65,0.5){\makebox(0,0)[c]{{\ovalBox(1.6,1.2){$\mc{R}_n$}}}}
\put(12.5,-0.5){\makebox(0,0)[c]{\tiny$\epsilon_{-\hf}-\epsilon_{\bar{1}}$}}
\put(18.5,-0.5){\makebox(0,0)[c]{\tiny$\epsilon_{\ov{k}}-\epsilon_{1}$}}
\end{picture}
\end{center}
\bigskip
\begin{center}
\hskip -3cm \setlength{\unitlength}{0.16in}
\begin{picture}(24,1)
\put(3,0.5){\makebox(0,0)[c]{${\mathcal D}(\SG^\diamond(m,n))$:}}
\put(9.25,0.5){\makebox(0,0)[c]{{\ovalBox(1.6,1.2){$\mc{L}^\diamond_m$}}}}
\put(10.05,0.5){\line(1,0){1.85}}
\put(12.35,0.5){\makebox(0,0)[c]{$\bigotimes$}}
\put(12.8,0.5){\line(1,0){1.85}}
\put(15.45,0.5){\makebox(0,0)[c]{{\ovalBox(1.6,1.2){$\mc{K}$}}}}
\put(16.25,0.5){\line(1,0){1.85}}
\put(18.55,0.5){\makebox(0,0)[c]{$\bigotimes$}}
\put(19.0,0.5){\line(1,0){1.85}}
\put(21.65,0.5){\makebox(0,0)[c]{{\ovalBox(1.6,1.2){$\ov{\mc{R}}_n$}}}}
\put(12.5,-0.5){\makebox(0,0)[c]{\tiny$\epsilon_{-\hf}-\epsilon_{\bar{1}}$}}
\put(18.5,-0.5){\makebox(0,0)[c]{\tiny$\epsilon_{\ov{k}}-\epsilon_\hf$}}
\end{picture}
\end{center}
\bigskip
For $m=n=\infty$, these become the Dynkin diagrams and simple
systems of $\G$, $\SG$, $\G^\diamond$ and $\SG^\diamond$ with
respect to the Borel subalgebras $\mf{b}$, $\ov{\mf{b}}$,
$\mf{b}^\diamond$ and $\ov{\mf{b}}^\diamond$, respectively. The
fundamental systems of $\G$, $\SG$, $\G^\diamond$ and $\SG^\diamond$
are denoted by ${\Pi}$,  $\ov{\Pi}$, ${\Pi}^\diamond$ and
$\ov{\Pi}^\diamond$,  respectively. These  fundamental systems of
$\DG(m,n)$, $\G(m,n)$, $\SG(m,n)$, $\G^\diamond(m,n)$ and
$\SG^\diamond(m,n)$ are denoted by  $\wt{\Pi}^\diamond(m,n)$,
${\Pi}(m,n)$,  $\ov{\Pi}(m,n)$, ${\Pi}^\diamond(m,n)$ and
$\ov{\Pi}^\diamond(m,n)$,  respectively.

\subsection{Levi subalgebras}\label{sec:Levi}

We fix an arbitrary subset $Y_0$ of $\{\alpha_{\ov{j}} \,\vert\,1\le
j\le k-1\}$. The set $\wt{Y}$ (respectively $Y$, $\ov{Y}$,
$Y^\diamond$, $\ov{Y}^\diamond$) denotes the union of $Y_0$ with the
subset of $\wt{\Pi}$ (respectively, of $\Pi$, $\ov{\Pi}$,
$\Pi^\diamond$, $\ov{\Pi}^\diamond$) consisting of elements of the
form $\epsilon_r-\epsilon_s$, with $r,s\in\wt{\I}\cap\hf\Z$
(respectively ${\I}\cap\hf\Z$, $\ov{\I}\cap\hf\Z$,
${\I}^\diamond\cap\hf\Z$, $\ov{\I}^\diamond\cap\hf\Z$) satisfying
$r<s\le 0$ or $0<r<s$. Let $\wt{\mf{l}}$ (respectively ${\mf{l}}$,
$\ov{\mf{l}}$, ${\mf{l}}^\diamond$, $\ov{\mf{l}}^\diamond$) be the
standard Levi subalgebra of $\wt{\G}$ (respectively $\G$, $\SG$,
$\G^\diamond$, $\SG^\diamond$) corresponding to the set $\wt{Y}$
(respectively $Y$, $\ov{Y}$, $Y^\diamond$, $\ov{Y}^\diamond$).  Let
$\wt{\mf{p}}=\wt{\mf{l}}+\wt{\mf{b}}$ (respectively
${\mf{p}}={\mf{l}}+{\mf{b}}$, $\ov{\mf{p}}=\ov{\mf{l}}+\ov{\mf{b}}$,
${\mf{p}}^\diamond={\mf{l}}^\diamond+{\mf{b}}^\diamond$,
$\ov{\mf{p}}^\diamond=\ov{\mf{l}}^\diamond+\ov{\mf{b}}^\diamond$) be
the corresponding parabolic subalgebra of $\wt{\G}$ (respectively
$\G$, $\SG$, $\G^\diamond$, $\SG^\diamond$) with nilradical
$\wt{\mf{u}}$ (respectively ${\mf{u}}$, $\ov{\mf{u}}$,
${\mf{u}}^\diamond$, $\ov{\mf{u}}^\diamond$) and opposite nilradical
$\wt{\mf{u}}_-$ (respectively ${\mf{u}}_-$, $\ov{\mf{u}}_-$,
${\mf{u}}^\diamond_-$, $\ov{\mf{u}}^\diamond_-$).

Denote the standard Borel subalgebras $\wt{\mf{b}}\cap{{\mf{l}}}$
(respectively $\wt{\mf{b}}\cap{\ov{\mf{l}}}$,
$\wt{\mf{b}}\cap{\mf{l}^\diamond}$,
$\wt{\mf{b}}\cap{\ov{\mf{l}}^\diamond}$) of ${\mf{l}}$ (respectively
$\ov{\mf{l}}$, ${\mf{l}}^\diamond$, $\ov{\mf{l}}^\diamond$) by
${\mf{b}}_{\mf{l}}$ (respectively  $\ov{\mf{b}}_{\mf{l}}$,
${\mf{b}}_{\mf{l}}^{\diamond }$, $\ov{\mf{b}}_{\mf{l}}^{\diamond}$).
For $\mu\in \wt{\mf{h}}^*$ (respectively ${\mf{h}}^*$,
$\ov{\mf{h}}^*$, ${\mf{h}}^{\diamond *}$, $\ov{\mf{h}}^{\diamond
*}$), let $L(\wt{\mf{l}},\mu)$ (respectively $L(\mf{l},\mu)$,
$L(\ov{\mf{l}},\mu)$, $L(\mf{l}^\diamond,\mu)$,
$L(\ov{\mf{l}}^\diamond,\mu)$) denote the highest weight irreducible
$\wt{\mf{l}}$-(respectively ${\mf{l}}$-, $\ov{\mf{l}}$-,
${\mf{l}}^\diamond$-, $\ov{\mf{l}}^\diamond$-)module of highest
weight $\mu$ with respect to the standard Borel. We extend
$L(\mf{l},\mu)$ to a $\mf{p}$-module by letting $\mf{u}$ act
trivially.  Define as usual the parabolic Verma module $\Delta(\mu)$
and its unique irreducible quotient $L(\mu)$ over $\G$:
\begin{align*}
\Delta(\mu):=\text{Ind}_{\mf{p}}^{\G}L(\mf{l},\mu), \qquad
\Delta(\mu) \twoheadrightarrow L(\mu).
\end{align*}
Similarly, we introduce the other four variants of parabolic Verma
and irreducible quotient modules with self-explanatory notations.

We shall also need the finite-rank counterparts of the above algebras. Let
$\wt{\mf{l}}(m,n)=\wt{\mf{l}}\cap \DG(m,n)$. Then $\wt{\mf{l}}(m,n)$
is a Levi subalgebra of $\DG(m,n)$. For $\mu\in \wt{\mf{h}}(m,n)^*$,
let $L(\wt{\mf{l}}(m,n),\mu)$ denote the highest weight irreducible
$\wt{\mf{l}}(m,n)$-module of highest weight $\mu$. As above, we can
define parabolic Verma $\wt{\G}(m,n)$-module
$\wt{\Delta}_{(m,n)}(\mu)$ and its irreducible quotient
$\wt{L}_{(m,n)}(\mu)$. In a completely parallel fashion, we have the
other four variants of Levi subalgebras, parabolic Verma modules,
and so on, with self-explanatory notations.

\subsection{The dominant weights}
\label{sec:hw}


Given a partition $\mu=(\mu_1,\mu_2,\ldots)$, let $\mu'$ denote its
conjugate partition. We also denote by $\theta(\mu)$ the modified
Frobenius coordinates of $\mu$:
\begin{equation*}
\theta(\mu)
:=(\theta(\mu)_{1/2},\theta(\mu)_1,\theta(\mu)_{3/2},\theta(\mu)_2,\ldots),
\end{equation*}
where
$$\theta(\mu)_{i-1/2}:=\langle\mu'_i-i+1\rangle, \quad
\theta(\mu)_i:=\langle\mu_i-i\rangle, \quad i\in\N.
$$
Here and below $\langle b\rangle:=\max\{b,0\}$ for all
$b\in\mathbb{R}$. Let $a,\la_{1}^0,\ldots,\la_{k}^0\in\C$, $\la^-$
and $\la^+$ be two partitions. Associated to the tuple
$\la=(a,\la_{1}^0,\ldots,\la_{k}^0;\la^-,\la^+)$, set
\begin{align*}
\La^+(\la)&:=\sum_{i=1}^{k}\la_{i}^0\epsilon_{\ov{i}}
+\sum_{j\in\N}\la^+_{j}\epsilon_{j} &
\ov{\La}^+(\la)&:=\sum_{i=1}^{k}\la_{i}^0\epsilon_{\ov{i}}
+\sum_{j\in\N}(\la^+)'_{j}\epsilon_{j-\hf},\\
{\La}^-(\la)&:=-\sum_{j\in\N}\la^-_{j}\epsilon_{-j+1}+ a\La_0,&
\ov{\La}^-(\la)&:= -\sum_{j\in\N}(\la^-)'_{j}\epsilon_{-j+\hf}+ a\La_0.
\end{align*}
The tuple $(a,\la_{1}^0,\ldots,\la_{k}^0;\la^-,\la^+)$ is said to
satisfy a {\em dominant condition} if
\begin{equation}\label{dominant}
\langle\sum_{i=1}^{k}\la_i^0\epsilon_{\ov{i}},
\alpha^\vee \rangle\in\Z_+, \quad {\rm for \,\,all}\,\,\, \alpha\in Y_0,
\end{equation}
 where $\alpha^\vee$
denotes the coroot of $\alpha$. Associated to such a dominant
tuple and $a\in \C$, we define the weights (which will be
called {\em dominant})
\begin{align*}
\wt{\la} &:=\sum_{i=1}^{k}\la_{i}^0\epsilon_{\ov{i}}
-\sum_{r\in\hf\N}\theta({\la^-}')_r\epsilon_{-r+\hf}
+ \sum_{r\in\hf\N}\theta(\la^+)_r\epsilon_r
+ a\La_0\in \wt{\h}^{*}\\
{\la}&:={\La}^-(\la)+\La^+(\la)\in {\h}^{*},\quad\quad\quad
\ov{\la}:=\La^-(\la)+\ov{\La}^+(\la)\in \ov{\h}^{*},\\
{\la}^\diamond&:=\ov{\La}^-(\la)+\La^+(\la)\in {\h}^{\diamond*},\quad\quad
\ov{\la}^\diamond:=\ov{\La}^-(\la)+\ov{\La}^+(\la)\in \ov{\h}^{\diamond*}.
\end{align*}
The number $a$ will be called the {\em level} of these weights.

We denote by $\wt{P}^+_a$ (respectively $P^+_a$, $\ov{P}^+_a$,
$P^{\diamond+}_a$, $\ov{P}^{\diamond+}_a$) the set of all dominant
weights of the form $\wt{\la}$ (respectively $\la$, $\ov{\la}$,
$\la^{\diamond}$, $\ov{\la}^{\diamond}$) with a fixed $a\in \C$.
Obviously, we have bijective maps between $\wt{P}^+_a$, $P^+_a$,
$\ov{P}^+_a$, $P^{\diamond+}_a$ and $\ov{P}^{\diamond+}_a$ given by
$\wt{\la}\leftrightarrow\la\leftrightarrow\ov{\la}
\leftrightarrow\la^{\diamond}\leftrightarrow\ov{\la}^{\diamond}$ for
$\la\in P^+_a$. Finally, we let $P^+=\bigcup_{a\in\C}P^+_a$.

\section{Change of highest weights for different Borel subalgebras}
\label{sec:borel}

In this section, using odd reflections, we will determine how a
highest weight for a highest weight module changes from a standard
Borel to another distinguished non-standard Borel subalgebra.

We shall briefly explain the effect of an odd reflection on the
highest weight of a highest weight irreducible module (cf., e.g.,
\cite[Lemma 1]{PS}, \cite[Lemma 1.4]{KW}). Fix a Borel subalgebra
$\mc{B}$ of a Lie superalgebra $\mc G$ with corresponding positive
system $\Phi_+(\mc{B})$. Let $\alpha$ be an isotropic odd simple
root and $\alpha^\vee$ be its corresponding coroot. Applying the odd
reflection with respect to $\alpha$ changes the Borel subalgebra
$\mc{B}$ into a new Borel subalgebra $\mc{B}^{\alpha}$ whose
corresponding positive system is
$\Phi_+(\mc{B}^{\alpha})=\{-\alpha\}\cup\Phi_+(\mc{B})\setminus\{\alpha\}$.

\begin{lem}  \label{lem:odd}
Let $\la$ be the highest weight with respect to $\mc{B}$ of an
irreducible module.  If $\langle\la,\alpha^\vee\rangle\not=0$, then
the highest weight of this irreducible module with respect to
$\mc{B}^{\alpha}$ is $\la-\alpha$.  If
$\langle\la,\alpha^\vee\rangle=0$, then the highest weight remains
to be $\la$.
\end{lem}

Even though it is unclear how the structures of Verma modules are
related via odd reflections, the relation of their characters is
simply described as follows.
\begin{lem}\label{char:verma:change}
Let $\Delta(\mc G, {\mc B},\la)$ and $\Delta(\mc G,
\mc{B}^{\alpha},\la)$ denote the $\mc G$-Verma modules of highest
weight $\la$ with respect to the Borel subalgebras $\mc{B}$ and
$\mc{B}^{\alpha}$, respectively. Then we have
\begin{align*}
\text{ch}\Delta(\mc G, {\mc B},\la)=\text{ch}\Delta(\mc G, {\mc B}^{\alpha},\la-\alpha).
\end{align*}
\end{lem}

\begin{proof}
Follows from the identity
$\Phi_+(\mc{B}^{\alpha})=\{-\alpha\}\cup\Phi_+(\mc{B})\setminus\{\alpha\}$.
\end{proof}

For $n\in\N$, we introduce  the following total orderings of the
subsets of $\wt{\I}$:
 {\allowdisplaybreaks
\begin{align}
&\cdots\prec -\frac{3}{2}-n\prec -1-n\prec -\hf-n\prec -n\prec \hf-n
\prec \frac{3}{2}-n\prec \cdots\prec -\frac{3}{2} \prec -\hf \label{order:c}\\
&\prec 1-n\prec \cdots\prec -2\prec -1\prec 0\prec \ov{1}\prec
\cdots\prec \ov{k}\prec 1\prec 2\prec \cdots\prec n\nonumber\\
&\prec \hf\prec \frac{3}{2}\prec \cdots\prec n+\hf\prec n+1\prec
n+\frac{3}{2}\prec n+2\prec \cdots,\nonumber\\
&\cdots\prec -\frac{3}{2}-n\prec -1-n\prec -\hf-n\prec -n\prec \hf-n
\prec \frac{3}{2}-n\prec \cdots\prec -\frac{3}{2}\prec -\hf \label{order:s}\\
&\prec 1-n\prec \cdots\prec -2\prec -1\prec 0\prec \ov{1}\prec \cdots\prec
\ov{k}\prec \hf\prec \frac{3}{2}\prec \cdots\prec n-\hf\nonumber\\
&\prec 1\prec 2\prec \cdots\prec n\prec n+\hf\prec n+1\prec n+\frac{3}{2}\prec \cdots, \nonumber\\
&\cdots\prec -\frac{3}{2}-n\prec -1-n\prec -\hf-n\prec -n\prec 1-n\prec
\cdots\prec -1\prec 0 \label{order:v}\\
&\prec \hf-n \prec \frac32-n\prec \cdots\prec -\frac32\prec -\hf \prec \ov{1}\prec
\cdots\prec \ov{k}\prec 1\prec 2\prec \cdots\prec n\nonumber\\
&\prec \hf\prec \frac{3}{2}\prec \cdots\prec n+\hf\prec n+1\prec n+\frac{3}{2}\prec
 n+2\prec \cdots,\nonumber\\
&\cdots\prec -\frac{3}{2}-n\prec -1-n\prec -\hf-n\prec -n\prec 1-n\prec
 \cdots\prec -1\prec 0 \label{order:vs}\\
&\prec \hf-n \prec \frac32-n\prec \cdots\prec -\frac32\prec -\hf \prec \ov{1}\prec
 \cdots\prec \ov{k}\prec \hf\prec \frac{3}{2}\prec \cdots\prec n-\hf\nonumber\\
&\prec 1\prec 2\prec \cdots\prec n\prec n+\hf\prec n+1\prec n+\frac{3}{2}\prec \cdots. \nonumber
\end{align}
 }
For any total ordering of $\wt{\I}$, there is a Borel subalgebra of
$\wt{\G}$ spanned by the Cartan subalgebra $\wt{\h}$ and elements of
the form $\widehat{E}_{rs}$ such that $r<s$ with respect to the
ordering. Conversely, any Borel subalgebra of $\wt{\G}$ containing
the Cartan subalgebra $\wt{\h}$ determines a unique total ordering
of $\wt{\I}$. Let $\wt{\mf{b}}^{c}(n)$ (respectively
$\wt{\mf{b}}^{s}(n)$, $\wt{\mf{b}}^{\diamond c}(n)$,
$\wt{\mf{b}}^{\diamond s}(n)$) be the Borel subalgebras with respect
to the ordering \eqnref{order:c} (respectively \eqnref{order:s},
\eqnref{order:v}, \eqnref{order:vs}) of $\wt{\I}$.

Two elements $a$ and $b$ in an ordered set are said to be {\em
adjacent} if there is no element $j$ in the set satisfying $a<j<b$
or $a>j>b$. For an odd simple root of the form
$\epsilon_r-\epsilon_s$ with $r,s\in\wt{\I}$ in the root system of a
Borel subalgebra $\mc{B}$ of $\DG$ containing the Cartan subalgebra
$\wt{\h}$, the numbers $r$ and $s$ are adjacent with respect to the
corresponding ordering. The ordering corresponding to the new Borel
subalgebra obtained by applying the odd reflection with respect to
$\epsilon_r-\epsilon_s$ is the same as the ordering corresponding to
the Borel subalgebra $\mc{B}$ except reversing the ordering of $r$
and $s$.

\begin{rem}
Any ordering preserving the orderings of positive integers, positive
half integers, non-positive integers and negative half integers, and
satisfying $i<a<j$ for $i\in-\hf\Z_+$, $a\in\mathbb{K}$ and
$j\in\hf\N$, can be obtained by a sequence of reversing the
orderings of two adjacent indices $r$ and $s$ with $r\in\Z$ and
$s\in\hf+\Z$ satisfying $r,s>0$ or else $r,s\le 0$. Thus, the Borel
subalgebra with respect to an ordering satisfying the conditions
above can be obtained by applying a sequence of odd reflections to
the standard Borel subalgebra $\wt{\mf{b}}$. Moreover, we can choose
a sequence of odd reflections in such a way that it leaves the sets
of roots of $\wt{\mf{u}}$ and $\wt{\mf{u}}_-$ invariant.
\end{rem}

Hence, the Borel subalgebras $\wt{\mf{b}}^{c}(n)$,
$\wt{\mf{b}}^{s}(n)$, $\wt{\mf{b}}^{\diamond c}(n)$ and
$\wt{\mf{b}}^{\diamond s}(n)$ can be obtained from $\wt{\mf{b}}$ by
applying sequences of odd reflections leaving the sets of roots of
$\wt{\mf{u}}$ and $\wt{\mf{u}}_-$ invariant. Let us spell out
precisely the sequence of odd reflections required to obtain the
Borel subalgebra $\wt{\mf{b}}^{s}(n)$ from $\wt{\mf{b}}$ leaving the
sets of roots of $\wt{\mf{u}}$ and $\wt{\mf{u}}_-$ invariant. This
process can be easily modified for the remaining cases. Starting
with the Dynkin diagram of $\wt{\G}$, we apply the following
sequence $\hf n(n-1)$ odd reflections. First we apply one odd
reflection corresponding to $\epsilon_1-\epsilon_{\frac{3}{2}}$,
then we apply two odd reflections corresponding to
$\epsilon_2-\epsilon_{\frac{5}{2}}$ and
$\epsilon_1-\epsilon_{\frac{5}{2}}$.  After that we apply three odd
reflections corresponding to $\epsilon_3-\epsilon_{\frac{7}{2}}$,
$\epsilon_2-\epsilon_{\frac{7}{2}}$, and
$\epsilon_1-\epsilon_{\frac{7}{2}}$, et cetera, until finally we
apply $n-1$ odd reflections corresponding to
$\epsilon_{n-1}-\epsilon_{n-\hf}$,
$\epsilon_{n-2}-\epsilon_{n-\hf},\ldots,\epsilon_1-\epsilon_{n-\hf}$.
The corresponding fundamental system of the resulting new Borel
subalgebra for $\DG$ is listed as follows:
\begin{center}
\hskip -1cm \setlength{\unitlength}{0.16in}
\begin{picture}(24,4)
\put(15.25,2){\makebox(0,0)[c]{$\bigotimes$}}
\put(17.4,2){\makebox(0,0)[c]{$\bigcirc$}}
\put(21.9,2){\makebox(0,0)[c]{$\bigcirc$}}
\put(24.4,2){\makebox(0,0)[c]{$\bigotimes$}}
\put(26.8,2){\makebox(0,0)[c]{$\bigotimes$}}
\put(13.1,2){\line(1,0){1.70}}
\put(15.7,2){\line(1,0){1.25}} \put(17.8,2){\line(1,0){0.9}}
\put(20.1,2){\line(1,0){1.4}}
\put(22.35,2){\line(1,0){1.6}}
\put(24.9,2){\line(1,0){1.5}}
\put(27.2,2){\line(1,0){0.6}}
\put(19.5,1.95){\makebox(0,0)[c]{$\cdots$}}
\put(28.7,1.95){\makebox(0,0)[c]{$\cdots$}}
\put(15,1){\makebox(0,0)[c]{\tiny $\epsilon_{n-\hf}-\epsilon_1$}}
\put(17.8,1.1){\makebox(0,0)[c]{\tiny $\beta_{1}$}}
\put(22,1){\makebox(0,0)[c]{\tiny $\beta_{n-1}$}}
\put(24.4,1){\makebox(0,0)[c]{\tiny $\alpha_{n}$}}
\put(27,1){\makebox(0,0)[c]{\tiny $\alpha_{n+\hf}$}}
\put(8.8,2){\line(1,0){1.7}}
\put(11.8,2){\makebox(0,0)[c]{{\ovalBox(2.6,1.2){$\ov{\mc{R}}_{n}$}}}}
\put(5.6,1.95){\makebox(0,0)[c]{{\ovalBox(1.6,1.2){$\mc{K}$}}}}
\put(6.4,2){\line(1,0){1.5}}
\put(8.2,1){\makebox(0,0)[c]{\tiny $\epsilon_{\ov{k}}-\epsilon_{\hf}$}}
\put(8.35,2){\makebox(0,0)[c]{$\bigotimes$}}
\put(3,2){\makebox(0,0)[c]{$\bigcirc$}}
\put(0.8,2){\makebox(0,0)[c]{$\bigotimes$}}
\put(-1.5,2){\makebox(0,0)[c]{$\bigotimes$}}
\put(1.2,2){\line(1,0){1.4}} \put(3.4,2){\line(1,0){1.4}}
\put(-1.1,2){\line(1,0){1.4}}
\put(-2.6,2){\line(1,0){0.6}}
\put(-3.3,1.95){\makebox(0,0)[c]{$\cdots$}}
\put(-1.7,1){\makebox(0,0)[c]{\tiny $\alpha_{-1}$}}
\put(0.9,1){\makebox(0,0)[c]{\tiny $\alpha_{-\hf}$}}
\put(3.1,1){\makebox(0,0)[c]{\tiny $\epsilon_0-\epsilon_{\ov{1}}$}}
\end{picture}
\end{center}

Now we apply the following sequence of $\hf n(n-1)$ odd reflections
to the Dynkin diagram above. First we apply one odd reflection
corresponding to $\epsilon_{-1}-\epsilon_{-1/2}$, then we apply two
odd reflections corresponding to $\epsilon_{-2}-\epsilon_{-3/2}$ and
$\epsilon_{-2}-\epsilon_{-1/2}$.  After that we apply three odd
reflections corresponding to $\epsilon_{-3}-\epsilon_{-5/2}$,
$\epsilon_{-3}-\epsilon_{-3/2}$ and $\epsilon_{-3}-\epsilon_{-1/2}$,
et cetera, until finally we apply $n-1$ odd reflections
corresponding to
$\epsilon_{1-n}-\epsilon_{3/2-n},\epsilon_{1-n}-\epsilon_{5/2-n},
\ldots,\epsilon_{1-n}-\epsilon_{-1/2}$. We obtain the Dynkin diagram
of $\wt{\mf{b}}^{s}(n)$:
\begin{center}
\hskip 3cm \setlength{\unitlength}{0.16in}
\begin{picture}(24,4)
\put(15.25,2){\makebox(0,0)[c]{$\bigotimes$}}
\put(17.8,1.95){\makebox(0,0)[c]{{\ovalBox(1.6,1.2){$\mc{R}_n$}}}}
\put(20.45,2){\makebox(0,0)[c]{$\bigotimes$}}
\put(22.7,2){\makebox(0,0)[c]{$\bigotimes$}}
\put(24.5,2){\makebox(0,0)[c]{$\cdots$}}
\put(13.1,2){\line(1,0){1.70}}
\put(15.7,2){\line(1,0){1.25}}
\put(18.6,2){\line(1,0){1.4}}
\put(20.9,2){\line(1,0){1.3}}
\put(23.1,2){\line(1,0){0.5}}
\put(15.5,1){\makebox(0,0)[c]{\tiny $\epsilon_{n-\hf}-\epsilon_1$}}
\put(20.5,1){\makebox(0,0)[c]{\tiny $\alpha_{n}$}}
\put(23.1,1){\makebox(0,0)[c]{\tiny $\alpha_{n+\hf}$}}
\put(12.3,2){\makebox(0,0)[c]{{\ovalBox(1.6,1.2){$\ov{\mc{R}}_{n}$}}}}
\put(7.1,1.95){\makebox(0,0)[c]{{\ovalBox(1.4,1.2){$\mc{K}$}}}}
\put(9.5,1){\makebox(0,0)[c]{\tiny $\epsilon_{\ov{k}}-\epsilon_{\hf}$}}
\put(9.7,2){\makebox(0,0)[c]{$\bigotimes$}}
\put(4.7,2){\makebox(0,0)[c]{$\bigcirc$}}
\put(-0.3,2){\makebox(0,0)[c]{$\bigotimes$}}
\put(-5.2,2){\makebox(0,0)[c]{$\bigotimes$}}
\put(-6.4,2){\line(1,0){0.7}}
\put(-4.8,2){\line(1,0){1.3}}
\put(-2.1,2){\line(1,0){1.3}}
\put(0.1,2){\line(1,0){1.3}}
\put(2.9,2){\line(1,0){1.4}}
\put(5.1,2){\line(1,0){1.3}}
\put(7.8,2){\line(1,0){1.4}}
\put(10.1,2){\line(1,0){1.4}}
\put(-7.1,1.95){\makebox(0,0)[c]{$\cdots$}}
\put(-5,1){\makebox(0,0)[c]{\tiny $\alpha_{-n}$}}
\put(4.9,1){\makebox(0,0)[c]{\tiny $\epsilon_0-\epsilon_{\ov{1}}$}}
\put(2.15,2){\makebox(0,0)[c]{{\ovalBox(1.4,1.2){${\mc{L}}_{n}$}}}}
\put(-2.8,1.95){\makebox(0,0)[c]{{\ovalBox(1.4,1.2){${\mc{L}}_{n}^\diamond$}}}}
\put(.2,1){\makebox(0,0)[c]{\tiny $\epsilon_{-\hf}-\epsilon_{1-n}$}}
\end{picture}
\end{center}

The following lemma is a variant of Lemma~\ref{lem:odd} in our setting.

\begin{lem}  \label{hwt odd}
Let $\alpha$ be an odd simple root of the form
$\epsilon_r-\epsilon_s$, $r,s\in\wt{\I}$ satisfying $r<s\le 0$ or
$0<r<s$, in the root system of a Borel subalgebra $\mc{B}$ of
$\wt{\mf{l}}$ containing the Cartan subalgebra $\wt{\h}$. Let $\mc
B^\alpha$ denote the new Borel subalgebra obtained by applying the
odd reflection with respect to $\alpha$ to the Borel subalgebra
$\mc{B}$. For $\mu\in \wt{\mf{h}}^*$, let $v$ be a $\mc B$-highest
weight vector in $L(\wt{\mf{l}},\mu)$.
\begin{enumerate}
\item
If $\mu(\widehat{E}_{rr})+\mu(\widehat{E}_{ss})= 0$, then
$L(\wt{\mf{l}},\mu)$ is an $\wt{\mf{l}}$-module of $\mc
B^\alpha$-highest weight $\la$  and $v$ is a $\mc B^\alpha$-highest
weight vector.

\item
If $\mu(\widehat{E}_{rr})+\mu(\widehat{E}_{ss}) \neq 0$, then
$L(\wt{\mf{l}},\mu)$ is an $\wt{\mf{l}}$-module of $\mc
B^\alpha$-highest weight $\la -\alpha$ and $\widehat{E}_{sr} v$ is a
$\mc B^\alpha$-highest weight vector.
\end{enumerate}
\end{lem}

Using similar arguments as in the proofs of \cite[Lemma 3.2]{CL} and
\cite[Corollary 3.3]{CL} (cf.~\cite[Proposition~4.3]{CLW}) together
with \lemref{hwt odd}, we have the following.

\begin{prop}\label{prop:change}
Let $\la=(a, \la_{1},\ldots,\la_{k};\la^-,\la^+)\in P^+_a$ and
$n\in\N$.
\begin{itemize}
\item[(i)]
Suppose that $(\la^-)'_{1}\le n$ and $ (\la^+)'_{1}\le n$. Then the
highest weight of $L(\wt{\mf{l}},\wt{\la})$ with respect to the
Borel subalgebra $\wt{\mf{b}}^{c}(n)\cap{\wt{\mf{l}}}$ is $\la$.
Furthermore, $\wt{\Delta}(\wt{\la})$ and $\wt{L}(\wt{\la})$ are
highest weight $\DG$-modules of highest weight $\la$ with respect to
the Borel  $\wt{\mf{b}}^{c}(n)$.

\item[(ii)]
Suppose that $(\la^-)'_{1}\le n$ and $\la^+_{1}\le n$. Then the
highest weight of $L(\wt{\mf{l}},\wt{\la})$ with respect to the
Borel subalgebra $\ov{\mf{b}}^{s}(n)\cap{\wt{\mf{l}}}$ is $\ov\la$.
Furthermore, $\wt{\Delta}(\wt{\la})$ and $\wt{L}(\wt{\la})$ are
highest weight $\DG$-modules of highest weight $\ov\la$ with respect
to the Borel  $\wt{\mf{b}}^{s}(n)$.

\item[(iii)]
Suppose that $\la^-_{1}\le n$ and $(\la^+)'_{1}\le n$. Then the
highest weight of $L(\wt{\mf{l}},\wt{\la})$ with respect to the
Borel subalgebra $\wt{\mf{b}}^{\diamond c}(n)\cap{\wt{\mf{l}}}$ is
$\la^{\diamond}$. Furthermore, $\wt{\Delta}(\wt{\la})$ and
$\wt{L}(\wt{\la})$ are highest weight $\DG$-modules of highest
weight $\la^{\diamond}$ with respect to the Borel
$\wt{\mf{b}}^{\diamond c}(n)$.

\item[(iv)]
Suppose that $\la^-_{1}\le n$ and $\la^+_{1}\le n$. Then the highest
weight of $L(\wt{\mf{l}},\wt{\la})$ with respect to the Borel
subalgebra $\ov{\mf{b}}^{\diamond s}(n)\cap{\wt{\mf{l}}}$ is
$\ov\la^{\diamond}$. Furthermore, $\wt{\Delta}(\wt{\la})$ and
$\wt{L}(\wt{\la})$ are highest weight $\DG$-modules of highest
weight $\ov\la^{\diamond}$ with respect to the Borel
$\wt{\mf{b}}^{\diamond s}(n)$.
\end{itemize}
\end{prop}

\section{Super duality} \label{sec:O}

In this section, we first introduce the module categories $\Ot$,
$\Oa$, $\Or$, $\Ol$ and $\Olr$ for the infinite-rank Lie
superalgebras $\wt{\G}$, $\G$, $\SG$, $\G^\diamond$, and
$\SG^\diamond$, respectively. We also introduce functors $T: \Ot
\rightarrow \Oa$, $\ov{T}: \Ot \rightarrow \Or$, $T^\diamond: \Ot
\rightarrow \Ol$ and $\ov{T}^\diamond: \Ot \rightarrow \Olr$, and
then show that they are equivalences of tensor categories which send
parabolic Verma and simple modules to parabolic Verma and simple
modules, respectively.

\subsection{BGG categories}

Let $\Otn$ be the category of $\wt{\G}$-modules $\wt{M}$ such that $\wt{M}$
is a semisimple $\wt{\h}$-module with finite-dimensional weight
subspaces $\wt{M}_\gamma$, $\gamma\in \wt{\h}^*$, satisfying
\begin{itemize}
\item[(i)] $\wt{M}$ decomposes over $\wt{\mf{l}}$ into a direct sum of
$L(\wt{\mf{l}},\wt{\mu})$ for $\mu\in {P^+}$.
\item[(ii)] There exist finitely many weights
$\la_1,\la_2,\ldots,\la_k\in{P^+}$ (depending on $\wt{M}$) such that
if $\gamma$ is a weight in $\wt{M}$, then
$\gamma\in\wt{\la}_i-\sum_{\alpha\in{\wt{\Pi}}}\Z_+\alpha$, for some $i$.
\end{itemize}
The morphisms in $\Otn$ are (not necessarily even)
$\wt{\G}$-homomorphisms. We recall that $\varphi\in {\rm
Hom}_{\Otn}(\wt{M},\wt{N})$ means that $\varphi(xv)=x\varphi(v)$, for $x\in\wt{\G}$ and $v\in\wt{M}$\footnote{We remark that a different definition of homomorphism is $\varphi(xv)=(-1)^{|\varphi|\cdot|x|}x\varphi(v)$. The difference is inessential, as the map $f\rightarrow f^*$ given by $f^*(v):=(-1)^{|f|}f(v)$ provides a bijection between these two types of maps.}.
Let $\Ot$ be the full subcategory of $\Otn$ consisting of the objects
$\wt{M}\in \Otn$ such that $Kv=av$ for all $v\in \wt{M}$. We have
$\Otn=\bigoplus_{a\in\C}\Ot$. The parabolic Verma modules $\wt{\Delta}(\wt{\mu})$ and
irreducible modules $\wt{L}(\wt{\mu})$ for $\mu\in P^+_a$ lie in $\Oa$, by
\lemref{lem:paraverma} below. Analogously we can define the other
four variants $\On$, $\Orn$, $\Oln$ and $\Olrn$ of the above categories for ${\G}$, $\SG$,
$\G^\diamond$, and $\SG^\diamond$ in self-expained notations, which
contain the corresponding parabolic Verma and irreducible modules.

Let ${\wt{\Gamma}}:=\{\mu\in {\wt{\h}}^*\mid \mu({E}_{rr})=0,
\,\,|r|\gg 0 \}$. For
$\wt{V}=\bigoplus_{\wt{\Gamma}}\wt{V}_\gamma\in \Ot$, there is a
natural $\Z_2$-gradation $\wt{V}=\wt{V}_{\bar{0}}\bigoplus
\wt{V}_{\bar{1}}$ given by
\begin{equation}\label{wt-Z2-gradation}
 \wt{V}_{\bar{0}}:=\bigoplus_{\mu\in{\wt{\Gamma}}_{\bar{0}}}\wt{V}_{\mu}\qquad\hbox{and}\qquad
 \wt{V}_{\bar{1}}:=\bigoplus_{\mu\in{\wt{\Gamma}}_{\bar{1}}}\wt{V}_{\mu},
\end{equation}
where ${\wt{\Gamma}}_\epsilon:=\{\mu\in {\wt{\Gamma}}\mid \sum_{r\in
{1/ 2}+\Z}\mu({E}_{rr})\equiv {\epsilon} \,\,(\text{mod }2)\}$
(cf.~\cite[Section 2.5]{CL}). Let $\Pi$ denote the parity change
functor on $\Otn$. We define $\Otb$ to be the full subcategories of
$\Ot$ consisting of objects with $\Z_2$-gradation given by
\eqnref{wt-Z2-gradation} and define $\Ot^{\bar{1}}$ to be the full
subcategories of $\Ot$ consisting of objects $\Pi \wt{M}$ with
$\wt{M}\in \Otb$. Note that the morphisms in $\Otb$ and
$\Ot^{\bar{1}}$ are of degree $\bar{0}$. It is clear that $\Otb$
and $\Ot^{\bar{1}}$ are an abelian categories.

For $\wt{M}=\wt{M}_{\bar{0}} +\wt{M}_{\bar{1}}\in \Otn$, let
$\wt{M}_{\bar{0},\mu}=\wt{M}_{\mu}\cap\wt{M}_{\bar{0}}$ and
$\wt{M}_{\bar{1},\mu}=\wt{M}_{\mu}\cap\wt{M}_{\bar{1}}$. Also we let
$$\wt{M}'_{\bar{0}}:=\bigoplus_{\mu\in{\wt{\Gamma}}_{\bar{0}}}\wt{M}_{\bar{0},\mu},
\quad \wt{M}'_{\bar{1}}:=\bigoplus_{\mu\in{\wt{\Gamma}}_{\bar{1}}}\wt{M}_{\bar{1},\mu};
\qquad \wt{M}''_{\bar{0}}:=\bigoplus_{\mu\in{\wt{\Gamma}}_{\bar{1}}}\wt{M}_{\bar{0},\mu},
\quad \wt{M}''_{\bar{1}}:=\bigoplus_{\mu\in{\wt{\Gamma}}_{\bar{0}}}\wt{M}_{\bar{1},\mu}.$$
It is clear that
$\wt{M}':=\wt{M}'_{\bar{0}}\bigoplus\wt{M}'_{\bar{1}}$ and
$\wt{M}'':=\wt{M}''_{\bar{0}}\bigoplus\wt{M}''_{\bar{1}}$ are
submodules of $\wt{M}$. Since $\Dh$ separates the $\wt{M}_\mu$'s, we
have $\wt{M}_\mu=\wt{M}_{\bar{0},\mu}\bigoplus\wt{M}_{\bar{1},\mu}$
for all $\mu$. Therefore $\wt{M}=\wt{M}'\bigoplus\wt{M}''$. Hence
for every $\wt{M}\in \Otn$, $\wt{M}$ is even isomorphic to
$\wt{N}'\bigoplus\Pi\wt{N}''$ for some $\wt{N}', \wt{N}''\in
\Otn^{\bar{0}}$. For $\wt{N}\in \Otn$, let
$p_{\wt{N}}:\wt{N}\longrightarrow \Pi\wt{N}$ be the parity-reversing isomorphism.
Therefore for every $\wt{M}\in \Otn$, $\wt{M}$ is
isomorphic to $\wt{N}$ for some $\wt{N}\in \Otnb$. This implies that
the kernel and the cokernel of ${\varphi}$ belong to $\Otn$ for
$\wt{M}$, $\wt{N}\in\Otn$, and $\varphi\in{\rm
Hom}_{\Otn}(\wt{M},\wt{N})$. Hence $\Otn$ is an abelian category.
Also $\Otn$ and $\Otnb$ have isomorphic skeletons and hence they are
equivalent categories. Note that for $\wt{M}, \wt{N}\in \Otnb$,
$\varphi\in{\rm Hom}_{\Otn}(\Pi\wt{M},\wt{N})$ and $\varphi'\in{\rm
Hom}_{\Otn}(\wt{M},\Pi\wt{N})$, we have $\varphi=\phi\circ
p_{\Pi\wt{M}}$ and $\varphi'=p_{\wt{N}}\circ\phi'$ for some $\phi,
\phi'\in{\rm Hom}_{\Otn}(\wt{M},\wt{N})$.

Analogously, $\Orb$, $\Olb$ and $\Olrb$ denote the respective full
subcategories of $\Or$, $\Ol$ and $\Olr$ consisting of objects with
$\Z_2$-gradations given by \eqnref{wt-Z2-gradation} such that $r$
are summed over half integers contained in the respective index
sets; they are abelian and are equivalent to $\Or$, $\Ol$ and
$\Olr$, respectively.

Using similar arguments as in the proof of \cite[Lemma 3.1]{CLW}, we
have the following lemma.

\begin{lem}\label{lem:paraverma}
Let $\mu\in P^+_a$. The restrictions to $\wt{\mf{l}}$ of the
$\DG$-modules $\wt{\Delta}(\wt{\mu})$ and $\wt{L}(\wt{\mu})$
decompose into direct sums of $L(\wt{\mf{l}},\wt{\nu})$ for $\nu\in
P^+_a$. In particular, $\wt{\Delta}(\wt{\mu}),
\wt{L}(\wt{\mu})\in\Ot$. Analogous results hold for the other four
variants of categories for $\G$, $\SG$, $\G^\diamond$, and
$\SG^\diamond$.
\end{lem}
\subsection{The functors $T$, $\ov{T}$, $T^\diamond$ and $\ov{T}^\diamond$}\label{Tfunctors}
\label{sec:T}

We shall introduce four functors $T: \Ot \rightarrow \Oa$, $\ov{T}:
\Ot \rightarrow \Or$, $T^\diamond: \Ot \rightarrow \Ol$ and
$\ov{T}^\diamond: \Ot \rightarrow \Olr$. We will describe in detail
the construction of $T: \Ot \rightarrow \Oa$, and the remaining
three variants can be treated similarly.

By definition, $\G$ is naturally a subalgebra of $\DG$, $\mf l$ is a
subalgebra of $\wt{\mf{l}}$, and $\h$ is a subalgebra of $\wt{\h}$.
Furthermore, we may view ${\h}^*$ as a subspace of $\wt{\h}^*$ by
regarding $\h$ as a direct summand of $\wt{\h}$ with respect to the
natural basis $E_{rr}$'s. Given a semisimple $\wt{\h}$-module
$\wt{M}=\bigoplus_{\gamma\in\wt{\h}^*}\wt{M}_\gamma$, we define
\begin{align*}
T(\wt{M}):= \bigoplus_{\gamma\in{{\h}^*}}\wt{M}_\gamma.
\end{align*}
Note that $T(\wt{M})$ is an $\h$-submodule of $\wt{M}$. It in
addition $\wt{M}=\bigoplus_{\gamma\in\wt{\h}^*}\wt{M}_\gamma$ is an
$\wt{\mf{l}}$-module, then $T(\wt{M})$ is an $\mf l$-submodule of
$\wt{M}$. Furthermore, if
$\wt{M}=\bigoplus_{\gamma\in\wt{\h}^*}\wt{M}_\gamma$ is a
$\DG$-module, then  $T(\wt{M})$ is a $\G$-submodule of $\wt{M}$.

The direct sum decomposition in $\wt{M}$ gives rise to the natural
projection
\begin{eqnarray*}
 T_{\wt{M}} :  \wt{M}  \longrightarrow T(\wt{M})
\end{eqnarray*}
which is an $\h$--module homomorphism. If in addition $\wt{M}$ is an
$\wt{\mf{l}}$-module, then $T_{\wt{M}}$ is also an $\mf l$-module
homomorphism. Furthermore, if $\wt{M}$ is a $\DG$-module, then
$T_{\wt{M}}$ is a $\G$-module homomorphism. If
$\wt{f}:\wt{M}\rightarrow \wt{N}$ is an $\wt{\h}$-homomorphism, then
the following induced map
\begin{eqnarray*}
T[\wt{f}] :  T(\wt{M}) \longrightarrow T(\wt{N})
\end{eqnarray*}
is an $\h$-module homomorphism. If in addition
$\wt{f}:\wt{M}\rightarrow \wt{N}$ is a $\wt{\mf{l}}$-homomorphism,
then $T[\wt{f}] $ is an $\mf l$-module homomorphism. Furthermore, if
$\wt{f}:\wt{M}\rightarrow \wt{N}$ is a ${\wt{\G}}$-homomorphism,
then $T[\wt{f}] $ is a $\G$-module homomorphism.

\begin{lem}\label{T:Levi}
For $\la\in P^+_a$, we have
\begin{eqnarray*}
\CD
&T\big{(}L(\wt{\mf{l}},\wt{\la})\big{)}=L({\mf{l}},{\la}), \qquad
&\ov{T}\big{(}L(\wt{\mf{l}},\wt{\la})\big{)}=L(\ov{\mf{l}},\ov{\la}),\\
&T^\diamond\big{(}L(\wt{\mf{l}},\wt{\la})\big{)}
=L({\mf{l}}^\diamond,{\la}^\diamond), \qquad
&\ov{T}^\diamond\big{(}L(\wt{\mf{l}},\wt{\la})\big{)}
=L(\ov{\mf{l}}^\diamond,\ov{\la}^\diamond).
 \endCD
\end{eqnarray*}
\end{lem}

\begin{proof}
We will prove
$\ov{T}\big{(}L(\wt{\mf{l}},\wt{\la})\big{)}=L(\ov{\mf{l}},\ov{\la})$,
and the other cases can be proved by the same argument. By
\propref{prop:change}, $L(\wt{\mf{l}},\wt{\la})$ contains a
$(\wt{\mf b}^s(n)\cap \ov{\mf{l}})$-highest weight vector
$v_{\ov{\la}}$ of highest weight $\ov{\la}$ for $n \gg 0$.  The
vector $v_{\ov{\la}}$ clearly lies in
$\ov{T}\big{(}L(\wt{\mf{l}},\wt{\la})\big{)}$, and it is a $(\ov{\mf
b}\cap \ov{\mf{l}})$-singular vector since $\ov{\mf{l}}\cap \ov{\mf
b}=\ov{\mf{l}}\cap\wt{\mf b}^s(n)$. Let $w$ be any weight vector in
$\ov{T}\big{(}L(\wt{\mf{l}},\wt{\la})\big{)}$. Taking $n$ large
enough, we can assume that  $w\in U(\DG(n,n)\cap\wt{\mf
l})v_{\ov{\la}}$. Note that the Dynkin diagram of $\DG(n,n)$ with
respect to  the Borel subalgebra
$\mc{B}=\DG(n,n)\cap\wt{\mf{b}}^{s}(n)$ is the following:
\begin{center}
\hskip 3cm \setlength{\unitlength}{0.16in}
\begin{picture}(24,4)
\put(15.25,2){\makebox(0,0)[c]{$\bigotimes$}}
\put(17.8,1.95){\makebox(0,0)[c]{{\ovalBox(1.6,1.2){$\mc{R}_n$}}}}
\put(13.1,2){\line(1,0){1.70}}
\put(15.7,2){\line(1,0){1.25}}
\put(15.5,1){\makebox(0,0)[c]{\tiny $\epsilon_{n-\hf}-\epsilon_1$}}
\put(12.3,2){\makebox(0,0)[c]{{\ovalBox(1.6,1.2){$\ov{\mc{R}}_{n}$}}}}
\put(7.1,1.95){\makebox(0,0)[c]{{\ovalBox(1.4,1.2){$\mc{K}$}}}}
\put(9.5,1){\makebox(0,0)[c]{\tiny $\epsilon_{\ov{k}}-\epsilon_{\hf}$}}
\put(9.7,2){\makebox(0,0)[c]{$\bigotimes$}}
\put(4.7,2){\makebox(0,0)[c]{$\bigcirc$}}
\put(-0.3,2){\makebox(0,0)[c]{$\bigotimes$}}
\put(-2.1,2){\line(1,0){1.3}}
\put(0.1,2){\line(1,0){1.3}}
\put(2.9,2){\line(1,0){1.4}}
\put(5.1,2){\line(1,0){1.3}}
\put(7.8,2){\line(1,0){1.4}}
\put(10.1,2){\line(1,0){1.4}}
\put(4.9,1){\makebox(0,0)[c]{\tiny $\epsilon_0-\epsilon_{\ov{1}}$}}
\put(2.15,2){\makebox(0,0)[c]{{\ovalBox(1.4,1.2){${\mc{L}}_{n}$}}}}
\put(-2.8,1.95){\makebox(0,0)[c]{{\ovalBox(1.4,1.2){${\mc{L}}_{n}^\diamond$}}}}
\put(.2,1){\makebox(0,0)[c]{\tiny $\epsilon_{-\hf}-\epsilon_{1-n}$}}
\end{picture}
\end{center}
Let $\mc{N}_-$ be the opposite nilradical of $\mc{B}\cap
\wt{\mf{l}}$.  Then $w\in U(\mc{N}_-)v_{\ov{\la}}$, and hence we see
that $w\in U(\ov{\mf l})v_{\ov\la}$. This shows that
$\ov{T}\big{(}L(\wt{\mf{l}},\wt{\la})\big{)}$ is a highest weight $
\ov{\mf{l}}$-module with respect to the Borel subalgebra
$\ov{\mf{l}}\cap \ov{\mf{b}}$ of highest weight $\ov{\la}$.

In order to complete the proof we need to show that
$\ov{T}\big{(}L(\wt{\mf{l}},\wt{\la})\big{)}$ is irreducible. Let
$w$ be a $\ov{\mf b}\cap\ov{\mf l}$-singular vector in
$\ov{T}\big{(}L(\wt{\mf{l}},\wt{\la})\big{)}$. Choose $q\gg 0$ with
$q\ge n$ such that $w\in U(\DG(q,q)\cap\ov{\mf{l}})v_{\ov{\la}}$.
Since the root vector corresponding to a simple root in
$\wt{\mf{b}}^{s}(q)$ either commutes with
$U(\DG(q,q)\cap\ov{\mf{l}})$ or annihilates $w$, $w$ is a
$\wt{\mf{b}}^{s}(q)$-singular vector and hence $w$ is a multiple of
$v_{\ov{\la}}$. Therefore
$\ov{T}\big{(}L(\wt{\mf{l}},\wt{\la})\big{)}$ is irreducible.
\end{proof}

\begin{lem}\label{T:tensor}
For $\mu, \la\in P^+_a$, we have
\begin{itemize}
\item[(i)] $T\big{(}L(\wt{\mf{l}},\wt{\mu})\otimes L(\wt{\mf{l}},\wt{\la})\big{)}
=L({\mf{l}},{\mu})\otimes L({\mf{l}},{\la})$,
\item[(ii)] $\ov{T}\big{(}L(\wt{\mf{l}},\wt{\mu})\otimes
 L(\wt{\mf{l}},\wt{\la})\big{)}=L(\ov{\mf{l}},\ov{\mu})\otimes L(\ov{\mf{l}},\ov{\la})$,
\item[(iii)] $T^\diamond\big{(}L(\wt{\mf{l}},\wt{\mu})\otimes L(\wt{\mf{l}},\wt{\la})\big{)}
=L({\mf{l}}^\diamond,{\mu}^\diamond)\otimes L({\mf{l}}^\diamond,{\la}^\diamond)$,
\item[(iv)] $\ov{T}^\diamond\big{(}L(\wt{\mf{l}},\wt{\mu})\otimes L(\wt{\mf{l}},\wt{\la})\big{)}
=L(\ov{\mf{l}}^\diamond,\ov{\mu}^\diamond)\otimes L(\ov{\mf{l}}^\diamond,\ov{\la}^\diamond)$.
\end{itemize}
\end{lem}

\begin{proof}
We will prove $T\big{(}L(\wt{\mf{l}},\wt{\mu})\otimes
L(\wt{\mf{l}},\wt{\la})\big{)} =L({\mf{l}},{\mu})\otimes
L({\mf{l}},{\la})$, and the other cases can be proved by using the
same argument. By \lemref{T:Levi}, we have
$$T\big{(}L(\wt{\mf{l}},\wt{\mu})\otimes
L(\wt{\mf{l}},\wt{\la})\big{)} \supseteq L({\mf{l}},{\mu})\otimes
L({\mf{l}},{\la}).$$ The lemma now follows by a comparison of the
characters on both sides of (i).
\end{proof}

\begin{prop}\label{SD:comp:cat}
The functors $T: \wt{\mc{O}} \rightarrow {\mc{O}}$, $\ov{T}: \wt{\mc{O}} \rightarrow
\ov{\mc{O}}$, $T^\diamond: \wt{\mc{O}} \rightarrow {\mc{O}^\diamond}$ and $\ov{T}^\diamond: \wt{\mc{O}}
\rightarrow \ov{\mc{O}}^\diamond$ are exact.
\end{prop}

\begin{proof}
In light of Lemma~\ref{T:Levi}, it suffices to show that
if $\wt{M}\in\wt{\mc{O}}$, then $T(\wt{M})$, $\ov{T}(\wt{M})$, $T^\diamond(\wt{M})$, and $\ov{T}^\diamond(\wt{M})$ lie in $\mc{O}$, $\ov{\mc{O}}$, $\mc{O}^\diamond$, and $\ov{\mc{O}}^\diamond$, respectively. We shall only show that $T(\wt{M})\in\mc{O}$, as the argument for the other cases is similar. Also below we will only prove the case $k\ge 1$, as the case $k=0$ is proved by slightly modifying the argument.

It will be convenient to define $\Pi(\mf{k})$, $\Pi(\wt{\mc L})$ and $\Pi(\wt{\mc R})$ to be the sets $\{\alpha_{\ov{1}},\ldots,\alpha_{\ov{k-1}}\}$, $\{\alpha_{-1/2},\alpha_{-1},\ldots\}$, and $\{\alpha_{1/2},\alpha_{1},\ldots\}$, respectively, for the remainder of the proof.

Let $\wt{M}\in\wt{\mc O}$. Then there exists
$_1\wt{\zeta},_2\wt{\zeta},\ldots,_r\wt{\zeta}$, with
$_1\zeta,\ldots,_r\zeta\in P^+$ such that any weight of $\wt{M}$ is
bounded above by some $_i\wt{\zeta}$. Ignoring the level recall that we have
$_j\zeta=(_j\zeta^0,_j\zeta^-,_j\zeta^+)$, where ${}^j\zeta^0=(_j\zeta_1^0,\ldots,_j\zeta^0_k)$, and ${}_j\zeta^+$ and ${}_j\zeta^-$ are
partitions with $|{}_j\zeta^+|=k_j$ and $|{}_j\zeta^-|=l_j$. For each $_j\zeta$, let $P_j$ be
the following finite subset of $\h^*$:
\begin{align*}
P_j:=\{(_j\zeta^0,\eta,\mu)\mid \mu,\eta\in\mc{P}\text{ with }|\mu|={k_j}\text{ and }|\eta|=l_j\}.
\end{align*}
Set $M:=T(\wt{M})$ and  $P(M):=\bigcup_{j=1}^r P_j$.

We claim that given any weight $\nu$ of $M$, there exists $\gamma\in
P(M)$ such that $\gamma-\nu\in\Z_+\Pi$.
It suffices to prove the claim for $\nu\in P^+$. Since $\nu$ is also
a weight of $\wt{M}$, we have $_i\wt{\zeta}-\nu\in\Z_+\wt{\Pi}$,
for some $i$. Thus
\begin{align*}
_i\wt{\zeta}-\nu=q(\ep_{0}-\ep_{\ov{1}})+p(\ep_{\ov{k}}-\ep_{1/2}) +{}\kappa^-+{}\kappa^++{}\kappa^0,
\end{align*}
where $p,q\in\Z_+$,
${}\kappa^0\in\sum_{\alpha\in\Pi(\mf{k})}\Z_+\alpha$,
${}\kappa^-\in\sum_{\beta\in\Pi(\wt{\mc{L}})}\Z_+\beta$, and ${}\kappa^+\in\sum_{\beta\in\Pi(\wt{\mc{R}})}\Z_+\beta$.
This implies that ${}\nu^-$ is a partition of size $l_i+q$ and ${}\nu^+$ is a partition of $k_i+p$, and
hence there exists $\gamma\in P(M)$ such that ${}\nu^+$ and $\nu^-$ are obtained
from the partitions ${}\gamma^+$ and $\gamma^-$ by adding $p$ and $q$ boxes, respectively. For every
such a box of ${}\nu^\pm$, we record the row number in which it was added to ${}\gamma^\pm$ in the
multisets ${}J^\pm$ with $|{}J^+|=p$ and $|{}J^-|=q$. Then we have
\begin{align*}
\nu=\gamma-{}^0\kappa-\sum_{j\in {}J^+}(\epsilon_{\ov{k}}-\epsilon_j)-\sum_{j\in {}J^-}(\epsilon_{-j+1}-\epsilon_{\ov{1}}),
\end{align*}
and hence $\nu<\gamma$. Thus, we conclude that $M\in\mc{O}$.
\end{proof}

\subsection{A character formula}

The following theorem can be regarded as a weak version of super
duality which is to be established in Theorem~\ref{thm:equivalence}.
Using similar arguments as in the proof of \cite[Theorem 4.5]{CLW},
we have the following.

\begin{thm}\label{matching:modules}
Let $\wt{M}\in \Ot$ and $\la\in P^+_a$. If $\wt{M}$ is a highest
weight $\wt{\G}$-module of highest weight $\wt{\la}$, then
$T(\wt{M})$, $\ov{T}(\wt{M})$, $T^\diamond(\wt{M})$ and
$\ov{T}^\diamond(\wt{M})$ are highest weight modules over ${\G}$,
$\ov{\G}$, ${\G}^\diamond$ and $\ov{\G}^\diamond$ of highest weights
$\la$, $\ov{\la}$, ${\la}^\diamond$ and $\ov{\la}^\diamond$,
respectively. Furthermore, we have
\begin{itemize}
\item[(i)]
$T\big{(}\wt{\Delta}(\wt{\la})\big{)}
 =\Delta(\la)$ and $  T\big{(}\wt{L}(\wt{\la})\big{)}=L(\la)$;
\item[(ii)]
 $\ov{T}\big{(}\wt{\Delta}(\wt{\la})\big{)}
 =\ov{\Delta}(\ov{\la})$ and $
 \ov{T}\big{(}\wt{L}(\wt{\la})\big{)}=\ov{L}(\ov{\la})$;
\item[(iii)]
$T^\diamond\big{(}\wt{\Delta}(\wt{\la})\big{)}
 =\Delta^\diamond(\la^\diamond)$
 and $T^\diamond\big{(}\wt{L}(\wt{\la})\big{)}=L^\diamond(\la^\diamond)$;
\item[(iv)]
 $\ov{T}^\diamond\big{(}\wt{\Delta}(\wt{\la})\big{)}
 =\ov{\Delta}^\diamond(\ov{\la}^\diamond)$ and $
 \ov{T}^\diamond\big{(}\wt{L}(\wt{\la})\big{)}=\ov{L}^\diamond(\ov{\la}^\diamond)$.
\end{itemize}
\end{thm}

By standard arguments, \thmref{matching:modules} implies the
following character formula.

\begin{thm}\label{character}
Let $\la\in P^+_a$, and write ${\rm ch}L(\la)=\sum_{\mu\in
P^+_a}a_{\mu\la}{\rm ch}\Delta(\mu)$, $a_{\mu\la}\in\Z$.  Then
\begin{itemize}
\item[(i)] ${\rm ch}\wt{L}(\wt{\la})=\sum_{\mu\in
P^+_a}a_{\mu\la}{\rm ch}\wt{\Delta}(\mu^\theta)$,
\item[(ii)] ${\rm ch}\ov{L}(\ov{\la})=\sum_{\mu\in
P^+_a}a_{\mu\la}{\rm ch}\ov{\Delta}(\ov{\mu})$.
\item[(iii)] ${\rm ch}{L}^\diamond({\la}^\diamond)=\sum_{\mu\in
P^+_a}a_{\mu\la}{\rm ch}{\Delta}^\diamond({\mu}^\diamond)$,
\item[(iv)] ${\rm ch}\ov{L}^\diamond(\ov{\la}^\diamond)=\sum_{\mu\in
P^+_a}a_{\mu\la}{\rm ch}\ov{\Delta}^\diamond(\ov{\mu}^\diamond)$.
\end{itemize}
\end{thm}


%
%
\subsection{Identification of Kazhdan-Lusztig polynomials}\label{sec:homology}

For a module $V$ over a Lie (super)algebra $\mc{G}$ and $j\in\Z_+$,
let $H_j(\mc{G};V)$ denote the $j$th homology group of $\mc{G}$ with
coefficient in $V$. For a precise definition of homology groups of
Lie superalgebras with coefficients in a module and a precise
formula for the boundary operator we refer the reader to \cite{T} or
\cite[Section 4]{CL}.

For $\wt{M}\in\Ot$ we denote by $M=T(\wt{M})$ and
$\ov{M}=\ov{T}(\wt{M})$, $M^\diamond=T^\diamond(\wt{M})$ and
$\ov{M}^\diamond=\ov{T}^\diamond(\wt{M})$. Using similar arguments
as in the proof of \cite[Theorem 4.10]{CLW}, we have the following.

\begin{thm}\label{matching:KL} We have, for $j\ge 0$,
\begin{itemize}
\item[(i)] $T(H_j(\wt{\mf{u}}_-;\wt{M}))\cong
H_j({\mf{u}}_-;{M})$, as $\mf{l}$-modules.
\item[(ii)] $\ov{T}(H_j(\wt{\mf{u}}_-;\wt{M}))\cong
H_j({\ov{\mf{u}}}_-;\ov{M})$, as ${\ov{\mf{l}}}$-modules.
\item[(iii)] $T^\diamond(H_j(\wt{\mf{u}}_-;\wt{M}))\cong
H_j({\mf{u}}_-^\diamond;{M^\diamond})$, as $\mf{l}^\diamond$-modules.
\item[(iv)] $\ov{T}^\diamond(H_j(\wt{\mf{u}}_-;\wt{M}))\cong
H_j({\ov{\mf{u}}}_-^\diamond;\ov{M}^\diamond)$, as ${\ov{\mf{l}}}^\diamond$-modules.
\end{itemize}
\end{thm}

Setting $\wt{M}=\wt{L}(\wt{\la})$ in \thmref{matching:KL} and using
\thmref{matching:modules} we obtain the following.

\begin{cor}\label{matching:KL1}
For $\la\in P^+_a$ and $j\ge 0$, we have
\begin{itemize}
\item[(i)]
$T\big{(}H_j(\wt{\mf{u}}_-;\wt{L}(\wt{\la}))\big{)}\cong
H_j({\mf{u}}_-;L(\la))$, as $\mf{l}$-modules.

\item[(ii)]
$\ov{T}\big{(}H_j(\wt{\mf{u}}_-;\wt{L}(\wt{\la}))\big{)}\cong
H_j({\ov{\mf{u}}}_-;\ov{L}(\ov{\la}))$, as
${\ov{\mf{l}}}$-modules.

\item[(iii)]
$T^\diamond\big{(}H_j(\wt{\mf{u}}_-;\wt{L}(\wt{\la}))\big{)}\cong
H_j({\mf{u}}_-^\diamond;{L^\diamond}({\la}^\diamond))$, as
$\mf{l}^\diamond$-modules.
\item[(iv)] $\ov{T}^\diamond\big{(}H_j(\wt{\mf{u}}_-;\wt{L}(\wt{\la}))\big{)}\cong
H_j({\ov{\mf{u}}}_-^\diamond;\ov{L}^\diamond(\ov{\la}^\diamond))$,
as ${\ov{\mf{l}}}^\diamond$-modules.
\end{itemize}
\end{cor}

We define the {\em parabolic Kazhdan-Lusztig polynomials} for the
categories $\Ot$, $\Oa$, $\Or$, $\Ol$ and $\Olr$ associated to
$\mu,\la\in P^+_a$ by letting
\begin{align*}
&\wt{\ell}_{\wt{\mu}\wt{\la}}(q)
 :=\sum_{n=0}^\infty\text{dim}_\C\Big{(}\text{Hom}_{\wt{\mf{l}}}\big{[}
 L(\wt{\mf{l}},\wt{\mu}),
 H_j\big{(}{\wt{\mf{u}}}_-;{L}(\wt{\la})\big{)} \big{]}\Big{)} (-q)^{-j},\\
&{\ell}_{\mu\la}(q) :=\sum_{j=0}^\infty\text{dim}_\C
\Big{(}\text{Hom}_{{\mf{l}}}\big{[} L({\mf{l}},\mu),
H_j\big{(}{{\mf{u}}}_-;{L}(\la)\big{)} \big{]}\Big{)} (-q)^{-j},\\
&\ov{\ell}_{\ov{\mu}\ov{\la}}(q)
 :=\sum_{j=0}^\infty\text{dim}_\C\Big{(}\text{Hom}_{\ov{\mf{l}}}\big{[}
 L(\ov{\mf{l}},\ov{\mu}),
 H_j\big{(}{\ov{\mf{u}}}_-;\ov{L}(\ov{\la})\big{)} \big{]}\Big{)} (-q)^{-j},\\
&{\ell}^\diamond_{\mu^\diamond\la^\diamond}(q) :=\sum_{j=0}^\infty\text{dim}_\C
\Big{(}\text{Hom}_{{\mf{l}}^\diamond}\big{[} L({\mf{l}^\diamond},\mu^\diamond),
H_j\big{(}{{\mf{u}}}^\diamond_-;{L}^\diamond(\la^\diamond)\big{)} \big{]}\Big{)} (-q)^{-j},\\
&\ov{\ell}^\diamond_{\ov{\mu}^\diamond\ov{\la}^\diamond}(q)
 :=\sum_{j=0}^\infty\text{dim}_\C\Big{(}\text{Hom}_{\ov{\mf{l}}^\diamond}\big{[}
 L(\ov{\mf{l}}^\diamond,\ov{\mu}^\diamond),
 H_j\big{(}{\ov{\mf{u}}}_-^\diamond;\ov{L}^\diamond(\ov{\la}^\diamond)\big{)}
  \big{]}\Big{)} (-q)^{-j}.
\end{align*}

By Vogan's homological interpretation of the Kazhdan-Lusztig
polynomials \cite[Conjecture 3.4]{V} and the Kazhdan-Lusztig
conjecture \cite{KL} (proved in \cite{BB, BK}), $\ell_{\mu\la}(q)$
coincides with the original definition in the setting of semisimple
Lie algebras and moreover $\ell_{\mu\la} (1) =a_{\mu\la}$ (cf.
\thmref{character}). The following reformulation of
\corref{matching:KL1}
 is a generalization of \thmref{character}.

\begin{thm}   \label{matching:KLpol}
For $\la,\mu\in P^+_a$, we have the following identification of
Kazhdan-Lusztig polynomials for the categories $\Ot$, $\Oa$, $\Or$,
$\Ol$ and $\Olr$:
$$
\wt{\ell}_{\wt{\mu}\wt{\la}}(q) ={\ell}_{\mu\la}(q) =
\ov{\ell}_{\ov{\mu}\ov{\la}}(q)={\ell}^\diamond_{\mu^\diamond\la^\diamond}(q)
=\ov{\ell}^\diamond_{\ov{\mu}^\diamond\ov{\la}^\diamond}(q).$$
\end{thm}

\subsection{Super duality} \label{sec:category}

\begin{lem}\label{g-extension} Let
\begin{equation*}
 0\longrightarrow M'\longrightarrow
M{\longrightarrow}M''\longrightarrow 0
\end{equation*}
be an exact sequence of $\DG$ (respectively $\G$, $\SG$, $\Gd$ and
$\SGd$)-modules which are semisimple over $\Dh$ (respectively $\h$,
$\Sh$, $\h^\diamond$ and $\Sh^\diamond$)  such that $M', M''\in \Ot$
(respectively $\Oa$, $\Or$, $\Ol$ and $\Olr$). Then $M$ also belongs
to $\Ot$ (respectively $\Oa$, $\Or$, $\Ol$ and $\Olr$).
\end{lem}

\begin{proof}
The statement for $\Oa$ is clear.  The statements for the other
cases follow by arguments, for example, as for \cite[Theorems 3.1
and 3.2]{CK}.
\end{proof}

The arguments in \cite[Section 5]{CLW} can be adapted to prove the
following theorem.
\begin{thm} [Super Duality] \label{thm:equivalence}
The functors $T:\Ot\rightarrow{\Oa}$, $\ov{T}: \Ot \rightarrow \Or$,
$T^\diamond: \Ot \rightarrow \Ol$ and $\ov{T}^\diamond: \Ot
\rightarrow \Olr$ are equivalences of categories.
\end{thm}

\begin{rem}
In an extreme yet interesting case, the Lie superalgebra $D(\DG)$ in
Section~\ref{sec:glhat}  (when $k=0$) is associated to the following
Dynkin diagram with all simple roots being odd; its subalgebra
$D(\G)$ and $D(\SG^\diamond)$ in Section~\ref{subalgebras} are Lie
algebras associated to the following Dynkin diagrams and fundamental
systems:
\begin{center}
\hskip -1cm \setlength{\unitlength}{0.16in}
\begin{picture}(24,3)
\put(-3,2){\makebox(0,0)[c]{$D(\DG)$:}}
\put(5.6,2){\makebox(0,0)[c]{$\bigotimes$}}
\put(8,2){\makebox(0,0)[c]{$\bigotimes$}}
\put(10.4,2){\makebox(0,0)[c]{$\bigotimes$}}
\put(12.8,2){\makebox(0,0)[c]{$\bigotimes$}}
\put(15.2,2){\makebox(0,0)[c]{$\bigotimes$}}
\put(17.6,2){\makebox(0,0)[c]{$\bigotimes$}}
\put(20,2){\makebox(0,0)[c]{$\bigotimes$}}
\put(3.65,2){\line(1,0){1.5}}
\put(6,2){\line(1,0){1.5}}\put(8.4,2){\line(1,0){1.5}}
\put(10.8,2){\line(1,0){1.5}} \put(13.2,2){\line(1,0){1.5}}
\put(15.6,2){\line(1,0){1.5}} \put(18,2){\line(1,0){1.5}}
\put(20.3,2){\line(1,0){1.5}}
\put(2.7,1.95){\makebox(0,0)[c]{$\cdots$}}
\put(23,1.95){\makebox(0,0)[c]{$\cdots$}}
\put(5.4,1){\makebox(0,0)[c]{\tiny $\alpha_{-\frac{3}{2}}$}}
\put(8,1){\makebox(0,0)[c]{\tiny $\alpha_{-1}$}}
\put(10.3,1){\makebox(0,0)[c]{\tiny $\alpha_{-\hf}$}}
\put(12.8,1){\makebox(0,0)[c]{\tiny $\alpha_{0}$}}
\put(15.2,1){\makebox(0,0)[c]{\tiny $\alpha_{\hf}$}}
\put(17.7,1){\makebox(0,0)[c]{\tiny $\alpha_{1}$}}
\put(20.2,1){\makebox(0,0)[c]{\tiny $\alpha_{\frac{3}{2}}$}}
\end{picture}
\end{center}
\begin{center}
\hskip -1cm \setlength{\unitlength}{0.16in}
\begin{picture}(24,3)
\put(-3,2){\makebox(0,0)[c]{$D(\G)$:}}
\put(5.6,2){\makebox(0,0)[c]{$\bigcirc$}}
\put(8,2){\makebox(0,0)[c]{$\bigcirc$}}
\put(10.4,2){\makebox(0,0)[c]{$\bigcirc$}}
\put(12.8,2){\makebox(0,0)[c]{$\bigcirc$}}
\put(15.2,2){\makebox(0,0)[c]{$\bigcirc$}}
\put(17.6,2){\makebox(0,0)[c]{$\bigcirc$}}
\put(20,2){\makebox(0,0)[c]{$\bigcirc$}}
\put(3.65,2){\line(1,0){1.5}}
\put(6,2){\line(1,0){1.5}}\put(8.4,2){\line(1,0){1.5}}
\put(10.8,2){\line(1,0){1.5}} \put(13.2,2){\line(1,0){1.5}}
\put(15.6,2){\line(1,0){1.5}} \put(18,2){\line(1,0){1.5}}
\put(20.3,2){\line(1,0){1.5}}
\put(2.7,1.95){\makebox(0,0)[c]{$\cdots$}}
\put(23,1.95){\makebox(0,0)[c]{$\cdots$}}
\put(5.4,1){\makebox(0,0)[c]{\tiny $\beta_{-3}$}}
\put(8,1){\makebox(0,0)[c]{\tiny $\beta_{-2}$}}
\put(10.3,1){\makebox(0,0)[c]{\tiny $\beta_{-1}$}}
\put(12.8,1){\makebox(0,0)[c]{\tiny $\beta_{0}$}}
\put(15.2,1){\makebox(0,0)[c]{\tiny $\beta_{1}$}}
\put(17.7,1){\makebox(0,0)[c]{\tiny $\beta_{2}$}}
\put(20.2,1){\makebox(0,0)[c]{\tiny $\beta_{3}$}}
\end{picture}
\end{center}
\begin{center}
\hskip -1cm \setlength{\unitlength}{0.16in}
\begin{picture}(24,3)
\put(-3,2){\makebox(0,0)[c]{$D(\SG^\diamond)$:}}
\put(5.6,2){\makebox(0,0)[c]{$\bigcirc$}}
\put(8,2){\makebox(0,0)[c]{$\bigcirc$}}
\put(10.4,2){\makebox(0,0)[c]{$\bigcirc$}}
\put(12.8,2){\makebox(0,0)[c]{$\bigcirc$}}
\put(15.2,2){\makebox(0,0)[c]{$\bigcirc$}}
\put(17.6,2){\makebox(0,0)[c]{$\bigcirc$}}
\put(20,2){\makebox(0,0)[c]{$\bigcirc$}}
\put(3.65,2){\line(1,0){1.5}}
\put(6,2){\line(1,0){1.5}}\put(8.4,2){\line(1,0){1.5}}
\put(10.8,2){\line(1,0){1.5}} \put(13.2,2){\line(1,0){1.5}}
\put(15.6,2){\line(1,0){1.5}} \put(18,2){\line(1,0){1.5}}
\put(20.3,2){\line(1,0){1.5}}
\put(2.7,1.95){\makebox(0,0)[c]{$\cdots$}}
\put(23,1.95){\makebox(0,0)[c]{$\cdots$}}
\put(5.4,1){\makebox(0,0)[c]{\tiny $\beta_{-\frac{7}{2}}$}}
\put(8,1){\makebox(0,0)[c]{\tiny $\beta_{-\frac{5}{2}}$}}
\put(10.3,1){\makebox(0,0)[c]{\tiny $\beta_{-\frac{3}{2}}$}}
\put(12.8,1){\makebox(0,0)[c]{\tiny $\beta_{-\hf}$}}
\put(15.2,1){\makebox(0,0)[c]{\tiny $\beta_{\frac{1}{2}}$}}
\put(17.7,1){\makebox(0,0)[c]{\tiny $\beta_{\frac{3}{2}}$}}
\put(20.2,1){\makebox(0,0)[c]{\tiny $\beta_{\frac{5}{2}}$}}
\end{picture}
\end{center}
In the case when $\ell\in\N$, we let $\mc F$ be the (semisimple) subcategory of $\mc O_\ell$ of integrable $\glinf$-modules of level $\ell$. The corresponding
subcategory of $\overline{\mc O}_\ell^\diamond$ is a
suitable subcategory $\mc U$ of $\glinf$-modules of negative level $-\ell$.
By Theorem~\ref{thm:equivalence}, $\mc U$ is equivalent to $\mc F$,
and hence it is semisimple; moreover, we
recover the character formulas for modules in $\mc U$ in \cite{KR} and \cite{CL0} that
were obtained by completely different methods.
\end{rem}

\subsection{Equivalence of tensor categories}

Recall that any modules in $\Otb$, $\Orb$, $\Olb$ and $\Olrb$ are equipped natural
$\Z_2$-gradations given by \eqnref{wt-Z2-gradation}. Note that the functors $T$, $\ov{T}$, $T^\diamond$ and $\ov{T}^\diamond$ preserve the $\Z_2$-gradations of the modules. For $\wt{M}$, $\wt{N}\in \Otb$, we have  $\wt{M}\otimes \wt{N}\in \Otb$ by \cite[Theorem 3.2]{CK} and hence $\Otb$ is a commutative tensor category (see, for example, \cite[Section 4.2]{KS}) with the trivial module as the unit object. Similarly, ${\On}$, $\Orb$, $\Olb$ and $\Olrb$ are commutative tensor categories with the unit objects (cf. \cite[Theorem 3.1, 3.2]{CK}).

By \lemref{T:tensor} and \thmref{thm:equivalence}, we have the following theorem.

\begin{thm}\label{thm:equivalence:tensor}
The functors $T:\Otnb\rightarrow{\On}$, $\ov{T}: \Otnb
\rightarrow \Ornb$, $T^\diamond: \Otnb \rightarrow \Olnb$, and
$\ov{T}^\diamond: \Otnb \rightarrow \Olrnb$ are equivalences of
tensor categories.
\end{thm}

\section{Tilting modules}
\label{sec:SDtilting}

For $\wt{M}\in\Otb$ such that $\wt{M}=\sum_{\mu\in\Dh}\wt{M}_\mu$,
we consider the restricted dual
\begin{align*}
\wt{M}^\vee:=\{f\in \text{Hom}_\C(\wt{M},\C)\mid
f(\wt{M}_\gamma)=0,\text{ for all but finitely many
}\gamma\in\wt{\h}^*\},
\end{align*}
which is equipped with the usual $\Z_2$-gradation and $\DG$-action.
We twist the usual $\DG$-module structure on $\wt{M}^\vee$ with the
automorphism $\tau$ defined in \eqnref{tau}, and denote the
resulting $\DG$-module by $\wt{M}^\tau$. The $\Z_2$-gradation on
$\wt{M}^\tau$ equals the $\Z_2$-gradation defined by
\eqnref{wt-Z2-gradation}. There are natural isomorphisms
$(\wt{M}^\tau)^\tau=\wt{M}$. It is easy to see that ${\rm
ch}(\wt{M}^\tau)={\rm ch}(\wt{M})$ and
$\wt{L}(\wt{\la})^\tau=\wt{L}(\wt{\la})$. The restriction of $\tau$
to $\wt{\mf l}$ is also denoted by $\tau$. Similarly, we can define
$L(\wt{\mf l},\wt{\mu})^\tau$ for $\mu\in P_a^+$. Note that
$\wt{M}^\tau\in\Otb$ for all $\wt{M}\in\Otb$ since $L(\wt{\mf
l},\wt{\mu})^\tau\cong L(\wt{\mf l},\wt{\mu})$ for all $\mu\in P_a^+$. The
contravariant functor from $\Otb$ to itself defined by sending
$\wt{M}$ to $\wt{M}^\tau$ is also denoted by $\tau$. Note that the
functor $\tau$ is an isomorphism of categories.

The restrictions of the automorphism $\tau$ of $\DG$ to $\G$, $\SG$,
$\G^\diamond$ and $\SG^\diamond$ are also denoted by $\tau$.
Similarly, we can define $M^\tau$ for an object $M$ in $\Oa$,
$\Orb$, $\Olb$ and $\Olrb$. We show that $M^\tau$ belongs to the
same category as $M$  and $(M^\tau)^\tau=M$. Also we have ${\rm
ch}({M}^\tau)={\rm ch}({M})$ and $L^\tau=L$ for all irreducible
highest weight modules $L$.

We summarize the above discussion in the following.

\begin{prop}  \label{prop:tau}
We have ${\mf T}\circ\tau=\tau\circ{\mf T}$ for ${\mf T}=T, \ov{T},
T^\diamond$ or $\ov{T}^\diamond$.
\end{prop}


An object $M \in \Oa$ is said to have a {\em Verma flag}
if it has a (possibly infinite) increasing filtration of $\G$-modules:
$$0 =M_0 \subseteq M_1 \subseteq M_2 \subseteq \cdots$$
such that $M=\cup_{i\ge 0} M_i$ and each $M_i/M_{i-1}$ is either $0$
or isomorphic to a parabolic Verma module $\Delta(\la^{i})$ for some
$\la^{i} \in P^+$. We define $(M :\Delta(\mu))$ for $\mu\in P^+$ to
be the number of subquotients of a Verma flag of $M$ that are
isomorphic to $\Delta(\mu)$. The notion of Verma flag is defined in
a similar fashion in the categories $\Otnb$, $\Ornb$, $\Olnb$ and
$\Olrnb$.
%
A {\em tilting module} associated to $\la\in P^+$ in the
category $\On$ is an indecomposable $\G$-module $U(\la)$
such that
\begin{itemize}
\item[(1)]
$U(\la)$ has a Verma flag with $\Delta(\la)$ at the bottom, and
\item[(2)]
${\rm Ext}^1(\Delta(\mu),U(\la))=0$ for all $\mu\in P^+$.
\end{itemize}
For $\mu\in P^+$, the tilting module $\wt{U}(\wt{\mu})$
(respectively $\ov{U}(\ov\mu)$, ${U}^\diamond(\mu^\diamond)$,
$\ov{U}^\diamond(\ov\mu^\diamond)$) in $\Otnb$ (respectively
$\Ornb$, $\Olnb$, $\Olrnb$) are defined in a similar fashion.

From the proposition below, the tilting module $\wt{U}(\wt{\mu})$
(respectively $U(\la)$, $\ov{U}(\ov\mu)$,
${U}^\diamond(\mu^\diamond)$, $\ov{U}^\diamond(\ov\mu^\diamond)$)
belongs to $\Otb$ (respectively $\Oa$, $\Orb$, $\Olb$, $\Olrb$) for
$\mu\in P^+_a$.

\begin{prop} \label{tiltingO}
For $\la\in P^+_a$, there exists a unique (up to isomorphism)
tilting module $U(\la)$ associated to $\la$ in $\Oa$. There also
exists a unique (up to isomorphism) tilting module
$\wt{U}(\wt{\la})$ (respectively $\ov{U}(\ov\la)$,
${U}^\diamond(\la^\diamond)$, $\ov{U}^\diamond(\ov\la^\diamond)$)
associated to $\wt{\la}$ (respectively $\ov\la$, $\la^\diamond$,
$\ov\la^\diamond$) in $\Otb$ (respectively $\Orb$, $\Olb$, $\Olrb$).
Moreover, we have $T\big{(}U(\wt{\la})\big{)}=U({\la})$,
$\ov{T}\big{(}U(\wt{\la})\big{)}=U(\ov{\la})$,
$T^\diamond\big{(}U(\wt{\la})\big{)}=U({\la}^\diamond)$ and
$\ov{T}^\diamond\big{(}U(\wt{\la})\big{)}=U(\ov{\la}^\diamond)$ for
$\la\in P^+_a$, and $U^\tau=U$ for any tilting module $U$ in either
of these categories.
\end{prop}

\begin{proof}
The same proof for  \cite[Theorem~3.14]{CW} ensures the existence of
the tilting modules. From the construction of the tilting modules,
we have ${U}({\la})^\tau={U}({\la})$ for $\la\in P^+_a$.

The uniqueness of the tilting modules follows by adapting an
argument of Soergel \cite[Section 5]{So} (also cf. \cite{Br2}). As
the category $\Oa$ may not have enough injective objects, we need to
modify the proof therein as follows. First we remark that every
element in ${\rm End}(U(\la))$  for $\la\in P^+_a$ is either locally
nilpotent or an isomorphism, since $U(\la)$ is indecomposable and
all its weight space is finite dimensional and the usual argument
via the Fitting decomposition works.

Let $\mc C$ be an abelian category of $\G$-modules containing $\Oa$
with enough injective objects such that every module in $\mc C$ is a
semisimple ${\h}$-module. Let $U(\la)$ and $U'(\la)$ be two tilting
modules associated to $\la$ in $\Oa$. Then ${\rm
Ext}_\On^1(\Delta(\mu),U'(\la))=0$ for all $\mu\in P^+_a$. By
\lemref{g-extension}, ${\rm Ext}_{\mc C}^1(\Delta(\mu),U'(\la))=0$
for all $\mu\in P^+_a$. Since $U(\la)/\Delta(\la)$ has a Verma flag,
we can apply \cite[Lemma 5.10]{So} and conclude that ${\rm Ext}_{\mc
C}^1(U(\la)/\Delta(\la),U'(\la))=0$. Therefore we have
$$
0\longrightarrow {\rm
Hom}\big{(}U(\la)/\Delta(\la),U'(\la)\big{)}\longrightarrow {\rm
Hom}\big{(}U(\la),U'(\la)\big{)} \longrightarrow {\rm
Hom}\big{(}\Delta(\la),U'(\la)\big{)} \longrightarrow 0.$$
Hence there is $\varphi\in{\rm Hom}\big{(}U(\la),U'(\la)\big{)}$
which restricts to the identity map on $\Delta(\la)$. Similarly,
there is $\varphi'\in{\rm Hom}\big{(}U'(\la),U(\la)\big{)}$ which
restricts to the identity map on $\Delta(\la)$. Therefore
$\varphi\circ\varphi'$ and $\varphi'\circ\varphi$  are isomorphisms
since they are not locally nilpotent. Hence $U(\la)$ and $U'(\la)$
are isomorphic.

By \thmref{thm:equivalence:tensor}, \thmref{matching:modules} and
the first part of the proposition, tilting modules exist and are
unique in the categories $\Otb$, $\Orb$, $\Olb$ and $\Olrb$. The
last part of the theorem also follows from \propref{prop:tau},
\thmref{matching:modules} and ${U}({\la})^\tau={U}({\la})$ for
$\la\in P^+_a$.
\end{proof}

\begin{rem}
The argument above can be applied to show that  there exists a
unique tilting module associated to each dominant weight in the
categories $\mc{O}$, $\ov{\mc{O}}$, and $\wt{\mc{O}}$ defined in
\cite[Section 3.2]{CLW} as well.
\end{rem}

\section{Character formula for general linear Lie superalgebras}
\label{sec:charformula}

In this section we use the super duality and truncation functors to
obtain character formulas for irreducible modules of general linear
Lie superalgebras of finite rank.

\subsection{Module categories of finite-rank Lie superalgebras}
\label{sec:Of}

We recall that the finite-rank Lie superalgebras $\wt{\G}(m,n)$,
$\G(m,n)$, $\SG(m,n)$, $\G^\diamond(m,n)$, and $\SG^\diamond(m,n)$
are defined in \secref{subalgebras} for $m,n\in \Z_+$. In this
subsection, we introduce certain module categories of these
finite-rank Lie superalgebras, and relate them to the categories
studied earlier via the truncation functors (see
Proposition~\ref{truncation}).

Let $\Oanf$ be the full subcategory of $\G(m,n)$-modules $M$ such that $M$
is a semisimple $\h(m,n)$-module with finite-dimensional weight
subspaces $M_\gamma$, $\gamma\in {\h}(m,n)^*$, satisfying
\begin{itemize}
\item[(i)] ${M}$ decomposes over $\mf{l}(m,n)$ into a direct sum of
$L(\mf{l}(m,n),\mu)$ for $\mu\in {P^+}$.
\item[(ii)] There exist finitely many weights
$\la_1,\la_2,\ldots,\la_k\in{P^+}$ (depending on ${M}$) such that
if $\gamma$ is a weight in ${M}$, then
$\gamma\in\la_i-\sum_{\alpha\in{\Pi}(m,n)}\Z_+\alpha$, for some $i$.
\end{itemize}
Recall the central element $K$ in the Lie superalgebra $\G(m,n)$.
Let $\Oaf$ be the full subcategory of $\Oanf$ consisting of the
objects $M\in \Oanf$ such that $Kv=av$ for all $v\in M$. We
certainly have $\Oanf=\bigoplus_{a\in\C}\Oaf$. The parabolic Verma
modules $\Delta_{(m,n)}(\mu)$ and irreducible modules
$L_{(m,n)}(\mu)$ for $\mu\in P^+_a \cap \h (m,n)^*$ lie in
$\Oa(m,n)$, by \propref{truncation} below. The categories $\Oa(m,n)$
are abelian. Analogously we have four
variants of the above categories with similar properties for Lie
superalgebras $\wt{\G}(m,n)$, $\SG(m,n)$, $\G^\diamond(m,n)$, and
$\SG^\diamond(m,n)$ in self-explanatory notations. For $\wt{V} \in \Ot(m,n)$, we define a natural
$\Z_2$-gradation $\wt{V} =\wt{V}_{\bar 0} \oplus \wt{V}_{\bar 1}$ similar to
\eqref{wt-Z2-gradation}. This allows us to define a subcategory
$\Ot^{\bar 0}(m,n)$ of $\Otf$ as before and analogously we have three variants of the above categories for Lie superalgebras $\SG(m,n)$, $\G^\diamond(m,n)$, and $\SG^\diamond(m,n)$.


Let $M$ be an $\Dh$-, $\h$-, $\Sh$-, $\hd$- or $\Sh$-module such
that $M=\bigoplus_{\gamma}M_\gamma$, where $\gamma$ runs over
$\gamma\in\sum_{i\in{\Z_+}}\C\epsilon_{-i}+\sum_{1\le j\le
k}\C\epsilon_{\bar{j}}+\sum_{j\in\N}\C\epsilon_j+\C\La_0$. For
$m,n\in\Z_+\cup\{\infty\}$, we consider the truncated vector space
$$
\mf{ tr}_{(m,n)}(M)=\bigoplus_{\nu}M_\nu,
$$
where the sum is over $\nu\in\sum_{-m<i\le
0}\C\epsilon_{-i}+\sum_{1\le j\le k}\C\epsilon_{\bar{j}}+\sum_{1\le
j<n+1}\C\epsilon_j+\C\La_0$.

\begin{lem}\label{truncation-Levi}
Let $m,n\in\Z_+\cup\{\infty\}$ and $\mu\in P^+_a$. Then,
$\mf{tr}_{(m,n)}\big{(}\wt{L}(\Dl,\wt{\mu})\big{)}$ is an
$\Dlf$-module, and
\begin{equation}
\mf{tr}_{(m,n)}\big{(}\wt{L}(\Dl,\wt{\mu})\big{)}=
\begin{cases}
\wt{L}(\Dlf,\wt{\mu}),\,\,\text{if }\langle\wt{\mu},
\widehat{E}_{jj}\rangle=0,\forall j\le -m, \forall j\ge n+1,\\
0,\quad\text{otherwise}.
\end{cases}
\end{equation}
Parallel results hold for the truncation of $L(\cl,\mu),
\ov{L}(\Sl,\ov{\mu}), L^\diamond(\ld,{\mu}^\diamond),
\ov{L}^\diamond(\Sld,\ov{\mu}^\diamond)$.
\end{lem}

As a consequence of the lemma above, we obtain exact functors $\mf{
tr}_{(m,n)}: \Otb \rightarrow \Otf$, $\mf{ tr}_{(m,n)}: \Oa
\rightarrow \Oaf$, $\mf{ tr}_{(m,n)}: \Orb \rightarrow \Orf$, $\mf{
tr}_{(m,n)}: \Olb \rightarrow \Olf$ and $\mf{ tr}_{(m,n)}: \Olrb
\rightarrow \Olrf$, by sending $M$ to $\mf{ tr}_{(m,n)}(M)$.

The following lemma can be proved by the same arguments as in
\cite[Lemma 3.2]{CLW} and \cite[Proposition 1.5]{Don}.

\begin{prop}\label{truncation}
Let $m,n\in\Z_+\cup\{\infty\}$, $\mu\in P^+_a$ and $X=L,\Delta, U$.
\begin{itemize}
\item[(i)] $\mf{tr}_{(m,n)}\big{(}\wt{X}(\wt{\mu})\big{)}=
\begin{cases}
\wt{X}_{(m,n)}(\wt{\mu}),\quad\text{if }\langle\wt{\mu},
\widehat{E}_{jj}\rangle=0,\forall j\le -m \;\&\;  j\ge n+1,\\
0,\quad\text{otherwise}.
\end{cases}$

\item[(ii)] $\mf{tr}_{(m,n)}\big{(}X(\mu)\big{)} =
\begin{cases}
X_{(m,n)}(\mu),\quad\text{if }
\langle\mu,\widehat{E}_{jj}\rangle=0,\forall j\le -m \;\&\;  j\ge n+1,\\
0,\quad\text{otherwise}.
\end{cases}$

\item[(iii)] $\mf{tr}_{(m,n)}\big{(}\ov{X}(\ov{\mu})\big{)}=
\begin{cases}
\ov{X}_{(m,n)}(\ov{\mu}),\quad\text{if }
\langle\ov{\mu},\widehat{E}_{jj}\rangle=0,\forall j\le -m \;\&\;   j\ge n+1,\\
0,\quad\text{otherwise}.
\end{cases}$

\item[(iv)] $\mf{tr}_{(m,n)}\big{(}X^\diamond({\mu}^\diamond)\big{)} =
\begin{cases}
X_{(m,n)}^\diamond({\mu}^\diamond),\quad\text{if }
\langle{\mu}^\diamond,\widehat{E}_{jj}\rangle=0,\forall j\le -m \;\&\;  j\ge n+1,\\
0,\quad\text{otherwise}.
\end{cases}$

\item[(v)] $\mf{tr}_{(m,n)}\big{(}\ov{X}^\diamond(\ov{\mu}^\diamond)\big{)}=
\begin{cases}
\ov{X}_{(m,n)}^\diamond(\ov{\mu}^\diamond),\quad\text{if }
\langle\ov{\mu}^\diamond,\widehat{E}_{jj}\rangle=0,\forall j\le -m \; \& \;  j\ge n+1,\\
0,\quad\text{otherwise}.
\end{cases}$
\end{itemize}
\end{prop}

\begin{rem}
Let $\la\in P^+$ and $m,n\in\Z_+\cup\{\infty\}$. Let $d$ denote the
boundary operators for the complexes of the homology groups
$H_\bullet(\wt{\mf{u}}_-;\wt{L}(\wt{\la}))$,
$H_\bullet(\mf{u}_-;L(\la))$,
$H_\bullet(\ov{\mf{u}}_-;\ov{L}(\ov{\la}))$,
$H_\bullet(\mf{u}^\diamond_-;\ov{L}({\la}^\diamond))$, and
$H_\bullet(\ov{\mf{u}}^\diamond_-;\ov{L}(\ov{\la}^\diamond))$. From
the formula of $d$  (see e.g.~\cite[(4.1)]{CL}) it is easy to see
that the truncation functors are compatible with $d$, i.e.,
\begin{align*}
d \circ\mf{tr}_{(m,n)} = \mf{tr}_{(m,n)}\circ d.
\end{align*}

It follows that truncation functors $\mf{tr}_{(m,n)}$ send these
$\mf u$-homology groups to the corresponding $\mf u$-homology
groups, and the corresponding module multiplicities of Levi
subalgebras (or equivalently the corresponding Kazhdan-Lusztig
polynomials) match under the truncation functors.
\end{rem}

\subsection{Character formulas for $\gl(k|2)$-modules in the BGG category}
\label{sec:char:gl(k2)}

Let $\la\in P^+$. Then Proposition \ref{truncation} and Theorem
\ref{character} together imply that, for any $m,n\in\Z_+$, the
character of the irreducible $\SG^\diamond(m|n)$-module
$\ov{L}_{m,n}^\diamond(\ov{\la}^\diamond)$ can be computed from
knowledge of the irreducible character of the $\G$-module $L(\la)$.
On the other hand $\text{ch}L(\la)$ is determined by the the
coefficients  $a_{\mu\la}=\ell_{\mu\la}(1)$, where the polynomials
$\ell_{\mu\la}(q)$ are type $A$ parabolic Kazhdan-Lusztig
polynomials (see Section \ref{sec:homology}).

Recalling the notations from Section \ref{sec:Of} we note that
$\SG^\diamond(1,0)\cong\gl(k|2)\oplus\C K$ (see Section
\ref{sec:glhat}), where $\gl(k|2)$ denotes the Lie subalgebra of
linear transformations on $\bigoplus_{i=1}^k\C v_{\ov i}\bigoplus\C
v_{-\hf}\bigoplus\C v_{\hf}$.  Recall that
$P^+=\bigcup_{a\in\C}P^+_a$. Now choose $Y_0=\emptyset$ so that
$\ov{\mf{l}}^\diamond_{1,0}$ is simply the Cartan subalgebra of
$\SG^\diamond(1,0)$. Now for a fixed $a\in\C$, $m,n\in\N$ and
$\la^0_j\in\Z$ we have
\begin{align*}
\la=a\La_0-\sum_{i=0}^{m-1}\epsilon_{-i}+\sum_{j}\la^0_j\epsilon_{\ov
j}+\sum_{l=1}^n\epsilon_l\in P^+_a.
\end{align*}
This corresponds to the dominant tuple
$(a,\la^0_1,\ldots,\la^0_k;\la^-,\la^+)$, with $\la^-=(1^m)$, and
$\la^+=(1^n)$. Thus, we have by Lemma~\ref{char:verma:change}
\begin{align*}
\ov{\la}^\diamond=a\La_0-m\epsilon_{-\hf}+\sum_{j}\la^0_j\epsilon_{\ov
j}+n\epsilon_\hf.
\end{align*}

Now let $\mu=p\epsilon_{-\hf}+\sum_{j=1}^k\la^0_j\epsilon_{\ov
j}+q\epsilon_\hf$ be any integral weight of $\gl(k|2)$. Suppose we
want to compute the character of the irreducible module with highest
weight $\mu$ with respect to the Borel associated to the following
fundamental system:
\begin{center}
\begin{equation}\label{Dynkin:gl1k1}
\hskip -3cm \setlength{\unitlength}{0.16in}
\begin{picture}(20,1)
\put(5.6,0){\makebox(0,0)[c]{$\bigotimes$}}
\put(8,0){\makebox(0,0)[c]{$\bigcirc$}}
\put(10.4,0){\makebox(0,0)[c]{$\bigcirc$}}
\put(14.85,0){\makebox(0,0)[c]{$\bigcirc$}}
\put(17.25,0){\makebox(0,0)[c]{$\bigcirc$}}
\put(19.75,0){\makebox(0,0)[c]{$\bigotimes$}}
\put(8.4,0){\line(1,0){1.55}} \put(10.82,0){\line(1,0){0.8}}
\put(13.2,0){\line(1,0){1.2}} \put(15.28,0){\line(1,0){1.45}}
\put(17.6,0){\line(1,0){1.65}}
\put(6,0){\line(1,0){1.6}}
\put(12.5,-0.1){\makebox(0,0)[c]{$\cdots$}}
\put(5.5,-1.3){\makebox(0,0)[c]{\tiny$\epsilon_{-\hf}-\epsilon_{\ov 1}$}}
\put(8,-1.3){\makebox(0,0)[c]{\tiny$\alpha_{\ov{1}}$}}
\put(10.3,-1.3){\makebox(0,0)[c]{\tiny$\alpha_{\ov{2}}$}}
\put(15,-1.3){\makebox(0,0)[c]{\tiny$\alpha_{\ov{k-2}}$}}
\put(17.2,-1.3){\makebox(0,0)[c]{\tiny$\alpha_{\ov{k-1}}$}}
\put(19.8,-1.3){\makebox(0,0)[c]{\tiny$\epsilon_{\ov k}-\epsilon_\hf$}}
\end{picture}\end{equation}
\vskip 0.5cm
\end{center}
First, by tensoring with an integral multiple of the determinant
representation, if necessary, we can assume that $q=0$ and $p\in\Z$.
Now choose $a=-p$. Then the weight $\ov{\la}^\diamond$ corresponding
to the dominant tuple
$\la=(-p,\la^0_1,\ldots,\la^0_k,\emptyset,\emptyset)$ restricts
precisely to the weight $\mu$ on $\gl(k|2)$. It follows that the
irreducible character of any integral highest weight with respect to
the Borel associated to the fundamental system  \eqref{Dynkin:gl1k1}
are determined by the usual Kazhdan-Lusztig polynomials of type $A$
Lie algebras.

Now consider the standard Borel subalgebra of $\gl(k|2)$
corresponding to the following fundamental system:
\begin{center}
\begin{equation}\label{Dynkin:glk2}
\hskip -2cm \setlength{\unitlength}{0.16in}
\begin{picture}(24,1)
\put(5.6,0){\makebox(0,0)[c]{$\bigcirc$}}
\put(8,0){\makebox(0,0)[c]{$\bigcirc$}}
\put(10.4,0){\makebox(0,0)[c]{$\bigcirc$}}
\put(14.85,0){\makebox(0,0)[c]{$\bigcirc$}}
\put(17.25,0){\makebox(0,0)[c]{$\bigotimes$}}
\put(19.75,0){\makebox(0,0)[c]{$\bigcirc$}}
\put(8.4,0){\line(1,0){1.55}} \put(10.82,0){\line(1,0){0.8}}
\put(13.2,0){\line(1,0){1.2}} \put(15.28,0){\line(1,0){1.45}}
\put(17.6,0){\line(1,0){1.65}}
\put(6,0){\line(1,0){1.6}}
\put(12.5,-0.1){\makebox(0,0)[c]{$\cdots$}}
\put(5.5,-1.3){\makebox(0,0)[c]{\tiny$\alpha_{\ov 1}$}}
\put(8,-1.3){\makebox(0,0)[c]{\tiny$\alpha_{\ov{2}}$}}
\put(10.3,-1.3){\makebox(0,0)[c]{\tiny$\alpha_{\ov{3}}$}}
\put(15,-1.3){\makebox(0,0)[c]{\tiny$\alpha_{\ov{k-1}}$}}
\put(17.4,-1.3){\makebox(0,0)[c]{\tiny$\epsilon_{k}-\epsilon_{-\hf}$}}
\put(20.5,-1.3){\makebox(0,0)[c]{\tiny$\epsilon_{-\hf}-\epsilon_\hf$}}
\end{picture}
\end{equation}
\vskip 0.5cm
\end{center}
Note that the Dynkin diagram in \eqref{Dynkin:gl1k1} is related to
the standard Dynkin diagram \eqref{Dynkin:glk2} by a sequence of odd
reflections with respect to the following odd roots:
\begin{align}\label{sequence:odd:ref}
\epsilon_{-\hf}-\epsilon_{\ov 1},\epsilon_{-\hf}-\epsilon_{\ov
2},\ldots, \epsilon_{-\hf}-\epsilon_{\ov k}.
\end{align}
Let us use $\Delta'(\gamma)$ and $L'(\gamma)$ to denote the
$\gl(k|2)$-Verma and irreducible modules of highest weight $\gamma$
with respect to the Borel associated to \eqref{Dynkin:glk2}. Now
suppose that we have for $\la\in P^+_a$
\begin{align*}
\text{ch}\ov{L}^\diamond(\ov{\la}^\diamond)=\sum_{\mu}a_{\mu\la}\text{ch}\ov{\Delta}^\diamond(\mu),
\end{align*}
where $\ov{\Delta}^\diamond(\mu)$ is the Verma module with respect
to the Borel in \eqref{Dynkin:gl1k1}. Let $[\ov{\la}^\diamond]'$ be
the highest weight obtained from $\ov{\la}^\diamond$ by applying the
sequence of odd reflections in \eqref{sequence:odd:ref} following
the prescription of Lemma \ref{hwt odd}. Furthermore, we observe
that
\begin{align*}
\text{ch}\ov{\Delta}^\diamond(\mu)
=\text{ch}\Delta'(\mu-k\epsilon_{-\hf}+\sum_{i=1}^k{\epsilon}_{\ov
i}).
\end{align*}
Thus, we conclude that
\begin{align*}
\text{ch}L'([\ov{\la}^\diamond]')=\sum_{\mu}a_{\mu\la}
\text{ch}\Delta'(\mu+\sum_{i=1}^k{\epsilon}_{\ov i}-k\epsilon_{-\hf}).
\end{align*}
Therefore, the Kazhdan-Lusztig polynomials of type $A$ also solve
the character problem for irreducible highest weight modules with
highest weights with respect to the standard Borel of $\gl(k|2)$.

\subsection{A variant of Brundan's conjecture}

For notations on Fock spaces below, we will refer to \cite{CWZ, CW}
(which differs somewhat from \cite{Br1}). Let $\VV$ with basis
$\{v_i\}_{i\in \Z}$ denote the natural $U_q(\gl_\infty)$-module and let
$\WW$ denote its dual module with basis $\{w_i\}_{i\in \Z}$. The
``Fock space" $\wedge^m \WW \otimes \VV^{\otimes k} \otimes \wedge^n
\WW$ admits a standard basis (using $v_i$'s and $w_i$'s). One can
define a bar involution on $\wedge^m \WW \otimes \VV^{\otimes k}
\otimes \wedge^n \WW$ which is compatible with the bar involution on
$U_q(\gl_\infty)$. This gives rise to the canonical basis and dual
canonical basis on (a completion of) $\wedge^m \WW \otimes
\VV^{\otimes k} \otimes \wedge^n \WW$ (cf. \cite{Lu, Br1}).

We denote by $\mc O_{\underline{m}|k|\underline{n}}$ the parabolic
BGG category of $\gl(k|m+n)$-modules of integral weights, which are
semisimple with respect to the Levi subalgebra $\gl(m) \oplus \mf
h_k \oplus \gl(n)$ (here $\mf h_k$ denotes the Cartan in $\gl(k)$),
with respect to the Borel associated to the fundamental system
below:
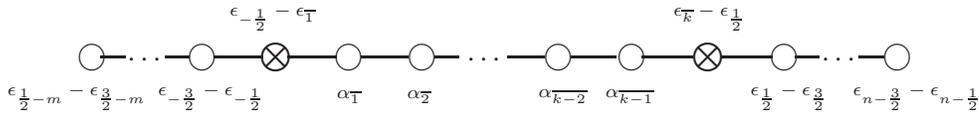
\begin{center}
\begin{equation*}
\hskip -3cm \setlength{\unitlength}{0.16in}
\begin{picture}(20,1)
\put(-0.4,0){\makebox(0,0)[c]{$\bigcirc$}}
\put(3.2,0){\makebox(0,0)[c]{$\bigcirc$}}
\put(5.6,0){\makebox(0,0)[c]{$\bigotimes$}}
\put(8,0){\makebox(0,0)[c]{$\bigcirc$}}
\put(10.4,0){\makebox(0,0)[c]{$\bigcirc$}}
\put(14.85,0){\makebox(0,0)[c]{$\bigcirc$}}
\put(17.25,0){\makebox(0,0)[c]{$\bigcirc$}}
\put(19.75,0){\makebox(0,0)[c]{$\bigotimes$}}
\put(22.25,0){\makebox(0,0)[c]{$\bigcirc$}}
\put(25.95,0){\makebox(0,0)[c]{$\bigcirc$}}
\put(-0.1,0){\line(1,0){0.8}}
\put(1.4,-0.1){\makebox(0,0)[c]{$\cdots$}}
\put(2.0,0){\line(1,0){0.8}} \put(3.6,0){\line(1,0){1.6}}
\put(6,0){\line(1,0){1.6}} \put(8.4,0){\line(1,0){1.55}}
\put(10.82,0){\line(1,0){0.8}}
\put(12.5,-0.1){\makebox(0,0)[c]{$\cdots$}}
\put(13.2,0){\line(1,0){1.2}} \put(15.28,0){\line(1,0){1.45}}
\put(17.6,0){\line(1,0){1.65}} \put(20.2,0){\line(1,0){1.65}}
\put(22.6,0){\line(1,0){0.8}}
\put(24.1,-0.1){\makebox(0,0)[c]{$\cdots$}}
\put(24.7,0){\line(1,0){0.8}}
\put(-0.9,-1.3){\makebox(0,0)[c]{\tiny$\epsilon_{\hf-m}-\epsilon_{\frac32-m}$}}
\put(3.5,-1.3){\makebox(0,0)[c]{\tiny$\epsilon_{-\frac32}-\epsilon_{-\hf}$}}
\put(5.5,1.3){\makebox(0,0)[c]{\tiny$\epsilon_{-\hf}-\epsilon_{\ov
1}$}}
 \put(8,-1.3){\makebox(0,0)[c]{\tiny$\alpha_{\ov{1}}$}}
\put(10.3,-1.3){\makebox(0,0)[c]{\tiny$\alpha_{\ov{2}}$}}
\put(15,-1.3){\makebox(0,0)[c]{\tiny$\alpha_{\ov{k-2}}$}}
\put(17.2,-1.3){\makebox(0,0)[c]{\tiny$\alpha_{\ov{k-1}}$}}
\put(19.8,1.3){\makebox(0,0)[c]{\tiny$\epsilon_{\ov
k}-\epsilon_\hf$}}
\put(22.4,-1.3){\makebox(0,0)[c]{\tiny$\epsilon_{\hf}-\epsilon_{\frac32}$}}
\put(26.6,-1.3){\makebox(0,0)[c]{\tiny$\epsilon_{n-\frac32}-\epsilon_{n-\hf}$}}
\end{picture}
\end{equation*}
\vskip 0.8cm
\end{center}
We shall prove a variant of Brundan's conjecture (see
\cite[Conjecture~4.32]{Br1}) which can be informally formulated as
follows. Note that the Borel subalgebra used here is not a standard
one in contrast to Brundan's original setting.

\begin{thm}  \label{th:BKL}
The category $\mc O_{\underline{m}|k|\underline{n}}$ categorifies
the $U_q(\gl_\infty)$-module $\wedge^m\WW \otimes \VV^{\otimes k}
\otimes \wedge^n \WW$, where the parabolic Verma, tilting and simple
$\gl(k|m+n)$-modules correspond to the standard, canonical, and dual
canonical basis (at $q=1$), respectively.
\end{thm}

The essence of Brundan's conjecture is that the entries of the
transition matrix between the (dual) canonical basis and the
standard basis on $\wedge^m\WW \otimes \VV^{\otimes k} \otimes
\wedge^n \WW$ should be interpreted as the Kazhdan-Lusztig
polynomials for the category $\mc
O_{\underline{m}|k|\underline{n}}$. Indeed our approach is powerful
enough to establish that these transition matrix entries coincide
with the Poincare polynomials for the corresponding $\mf u$-homology
groups defined in Section~\ref{sec:homology}. A precise formulation
of Theorem~\ref{th:BKL} and its variants with self-contained proofs
would take many pages as it would require us to repeat in our
setting much of the notations and constructions in \cite{Br1, CW}
among other things. Let us instead sketch a proof of
Theorem~\ref{th:BKL} below.

The space of semi-infinite $q$-wedges decomposes according to ``sectors'' labeled by $\Z$ (\cite{KMS}). For $a\in \Z$, let $\wedge^{\infty+a}\VV$ be the space of
semi-infinite $q$-wedges of sector $a$, which is isomorphic to
the highest weight $U_q(\gl_\infty)$-module $L(\widehat{\Lambda}_a)$
of highest weight being the $a$th fundamental weight. Similarly the
space $\wedge^{\infty-a}\WW$ of semi-infinite wedges of ``sector
$-a$" is also isomorphic to the $U_q(\gl_\infty)$-module
$L(\widehat{\Lambda}_a)$. The canonical isomorphism of
$U_q(\gl_\infty)$-modules $\wedge^{\infty+a} \VV \cong
\wedge^{\infty-a} \WW$ commutes with the bar involution. (This
extends the observation made in \cite{CWZ} for $a=0$.) This induces
an isomorphism of $U_q(\gl_\infty)$-modules
$$
\wedge^{\infty+a} \VV \otimes \VV^{\otimes k} \otimes \wedge^\infty
\VV \cong \wedge^{\infty-a} \WW \otimes \VV^{\otimes k} \otimes
\wedge^\infty \WW,
$$
which matches the standard, canonical, and dual canonical bases,
respectively. Based on this together with
Theorem~\ref{thm:equivalence} as well as a Fock space reformulation
(cf. e.g. \cite[Theorem~4.14]{CW} for a special case; the general
case is similar) of the Kazhdan-Lusztig conjecture for parabolic BGG
categories in type $A$ \cite{BB, BK, KL}, we establish a
version of Theorem~\ref{th:BKL} for $m=n=\infty$.
Theorem~\ref{th:BKL} for finite $m$ and $n$ follows from this
version for $m=n=\infty$ and the compatibility of the (dual)
canonical basis for varying $m,n$ (compare \cite[Corollary~
2.10]{CWZ}, \cite[Corollary~ 2.6]{CW}).

\begin{rem}
Note that Theorem~\ref{th:BKL} for $m=n=1$ solves the irreducible
character problem of the full BGG category of $\gl(k|2)$-modules,
which can also be viewed as a reformulation of the result in
Section~\ref{sec:char:gl(k2)}. On the other hand, it is easy to give a
formula for the irreducible characters in the category $\mc O_{\underline{m}|k|\underline{n}}$ for general $m$ and $n$, in the spirit of Section~\ref{sec:char:gl(k2)}, if
one is willing to introduce messier notations.
\end{rem}

\section{Some more variations}

The formulation and construction in this paper have variations which
unfortunately involve more complicated notations. In this section,
we will explain one such variation and its implication on $\mf
u$-homology computation.

For fixed $p,q\in\Z_+$, let $F=\{r\in\hf\Z\,\vert\,-p< r\le 0\,\,
{\rm or}\,\,\hf\le r\le q\}$. Set ${\mathbb J}={\mathbb{I}}\cup F$,
$\ov{\mathbb J}=\ov{\mathbb{I}}\cup F$, ${\mathbb
J}^\diamond={\mathbb{I}}^\diamond\cup F$ and $\ov{\mathbb
J}^\diamond=\ov{\mathbb{I}}^\diamond\cup F$. The subalgebra of
$\wt{\G}$ generated by $\widehat{E}_{rs}$ with $r,s\in {\mathbb{J}}$
(respectively $\ov{\mathbb{J}}$, ${\mathbb{J}}^\diamond$,
$\ov{\mathbb{J}}^\diamond$) is denoted by $\mc{G}$ (respectively
$\ov{\mc{G}}$, ${\mc{G}}^\diamond$ and $\ov{\mc{G}}^\diamond$).
Similar results in \secref{sec:borel} and \secref{sec:O} are still
valid if $\G$ (respectively $\SG$, $\G^\diamond$ and $\SG^\diamond$)
are replaced by $\mc{G}$ (respectively $\ov{\mc{G}}$,
${\mc{G}}^\diamond$ and $\ov{\mc{G}}^\diamond$). It is worth
pointing out that the sets $P^+_a$, $\ov{P}^+_a$, $P^{\diamond+}_a$
and $\ov{P}^{\diamond+}_a$ of all dominant weights are determined by
the choice of the orderings preserving the orderings of positive
integers, positive half integers, non-positive integers and negative
half integers and satisfying $i<b<j$ for $i\in-\hf\Z_+$,
$b\in\mathbb{K}$ and $j\in\hf\N$. We do not carry out any details
here.

Let us illustrate by an example. Define orderings of ${\mathbb
J}={\mathbb{I}}\cup F$, $\ov{\mathbb J}=\ov{\mathbb{I}}\cup F$,
${\mathbb J}^\diamond={\mathbb{I}}^\diamond\cup F$ and $\ov{\mathbb
J}^\diamond=\ov{\mathbb{I}}^\diamond\cup F$ respectively by
 {\allowdisplaybreaks
\begin{align*}
\cdots\prec -2\prec -1\prec 0\prec  -p+\hf\prec \cdots\prec
-{\frac{3}{2}}\prec -{\hf}\prec \ov{1}\prec \cdots\prec \ov{k}\\
\prec \hf\prec \frac{3}{2}\prec \cdots\prec q-\hf\prec 1\prec 2\prec
\cdots,
\end{align*}
\begin{align*}
\cdots\prec -2\prec -1\prec 0\prec  -p+\hf\prec
\cdots\prec -{\frac{3}{2}}\prec -{\hf}\prec \ov{1}\prec \cdots\prec \ov{k} \\
\prec 1\prec 2\prec \cdots \prec q\prec \hf\prec \frac{3}{2}\prec \cdots,
\end{align*}
\begin{align*}
\cdots\prec -{\frac{3}{2}}\prec -{\hf}\prec -p+1\prec
\cdots -2\prec -1\prec  0\prec \ov{1}\prec \cdots\prec \ov{k}\\ \prec
\hf\prec \frac{3}{2}\prec \cdots\prec q-\hf\prec 1\prec 2\prec \cdots,
\end{align*}
\begin{align*}
\cdots\prec -{\frac{3}{2}}\prec -{\hf}\prec -p+1\prec \cdots -2\prec
-1\prec  0\prec \ov{1}\prec \cdots\prec \ov{k}\\
\prec 1\prec 2\prec \cdots \prec q\prec \hf\prec \frac{3}{2}\prec \cdots.
\end{align*}
 }
The corresponding  fundamental systems are indicated in the
following Dynkin diagrams:
\bigskip
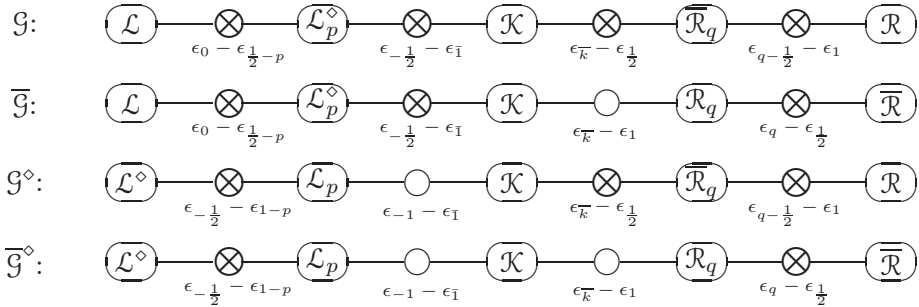
\begin{center}
\hskip -3cm \setlength{\unitlength}{0.16in}
\begin{picture}(24,1)
\put(-.5,0.5){\makebox(0,0)[c]{$\mc{G}$:}}
\put(3.0,0.5){\makebox(0,0)[c]{{\ovalBox(1.6,1.2){$\mc{L}$}}}}
\put(3.8,0.5){\line(1,0){1.85}}
\put(6.2,0.5){\makebox(0,0)[c]{$\bigotimes$}}
\put(6.6,0.5){\line(1,0){1.85}}
\put(9.25,0.5){\makebox(0,0)[c]{{\ovalBox(1.6,1.2){$\mc{L}_p^\diamond$}}}}
\put(10.05,0.5){\line(1,0){1.85}}
\put(12.35,0.5){\makebox(0,0)[c]{$\bigotimes$}}
\put(12.8,0.5){\line(1,0){1.85}}
\put(15.45,0.5){\makebox(0,0)[c]{{\ovalBox(1.6,1.2){$\mc{K}$}}}}
\put(16.25,0.5){\line(1,0){1.85}}
\put(18.55,0.5){\makebox(0,0)[c]{$\bigotimes$}}
\put(19.0,0.5){\line(1,0){1.85}}
\put(21.65,0.5){\makebox(0,0)[c]{{\ovalBox(1.6,1.2){$\ov{\mc{R}}_q$}}}}
\put(22.45,0.5){\line(1,0){1.85}}
\put(24.75,0.5){\makebox(0,0)[c]{$\bigotimes$}}
\put(25.2,0.5){\line(1,0){1.85}}
\put(27.85,0.5){\makebox(0,0)[c]{{\ovalBox(1.6,1.2){$\mc{R}$}}}}
\put(6.5,-0.5){\makebox(0,0)[c]{\tiny$\epsilon_0-\epsilon_{\hf-p}$}}
\put(12.5,-0.5){\makebox(0,0)[c]{\tiny$\epsilon_{-\hf}-\epsilon_{\bar{1}}$}}
\put(18.5,-0.5){\makebox(0,0)[c]{\tiny$\epsilon_{\ov{k}}-\epsilon_{\hf}$}}
\put(24.7,-0.5){\makebox(0,0)[c]{\tiny$\epsilon_{q-\hf}-\epsilon_{1}$}}
\end{picture}
\end{center}
\bigskip
\begin{center}
\hskip -3cm \setlength{\unitlength}{0.16in}
\begin{picture}(24,1)
\put(-.5,0.5){\makebox(0,0)[c]{$\ov{\mc{G}}$:}}
\put(3.0,0.5){\makebox(0,0)[c]{{\ovalBox(1.6,1.2){$\mc{L}$}}}}
\put(3.8,0.5){\line(1,0){1.85}}
\put(6.2,0.5){\makebox(0,0)[c]{$\bigotimes$}}
\put(6.6,0.5){\line(1,0){1.85}}
\put(9.25,0.5){\makebox(0,0)[c]{{\ovalBox(1.6,1.2){$\mc{L}_p^\diamond$}}}}
\put(10.05,0.5){\line(1,0){1.85}}
\put(12.35,0.5){\makebox(0,0)[c]{$\bigotimes$}}
\put(12.8,0.5){\line(1,0){1.85}}
\put(15.45,0.5){\makebox(0,0)[c]{{\ovalBox(1.6,1.2){$\mc{K}$}}}}
\put(16.25,0.5){\line(1,0){1.85}}
\put(18.55,0.5){\makebox(0,0)[c]{$\bigcirc$}}
\put(19.0,0.5){\line(1,0){1.85}}
\put(21.65,0.5){\makebox(0,0)[c]{{\ovalBox(1.6,1.2){$\mc{R}_q$}}}}
\put(22.45,0.5){\line(1,0){1.85}}
\put(24.75,0.5){\makebox(0,0)[c]{$\bigotimes$}}
\put(25.2,0.5){\line(1,0){1.85}}
\put(27.85,0.5){\makebox(0,0)[c]{{\ovalBox(1.6,1.2){$\ov{\mc{R}}$}}}}
\put(6.5,-0.5){\makebox(0,0)[c]{\tiny$\epsilon_0-\epsilon_{\hf-p}$}}
\put(12.5,-0.5){\makebox(0,0)[c]{\tiny$\epsilon_{-\hf}-\epsilon_{\bar{1}}$}}
\put(18.5,-0.5){\makebox(0,0)[c]{\tiny$\epsilon_{\ov{k}}-\epsilon_{1}$}}
\put(24.7,-0.5){\makebox(0,0)[c]{\tiny$\epsilon_{q}-\epsilon_{\hf}$}}
\end{picture}
\end{center}
\bigskip
\begin{center}
\hskip -3cm \setlength{\unitlength}{0.16in}
\begin{picture}(24,1)
\put(-.5,0.5){\makebox(0,0)[c]{${\mc{G}}^\diamond$:}}
\put(3.0,0.5){\makebox(0,0)[c]{{\ovalBox(1.6,1.2){$\mc{L}^\diamond$}}}}
\put(3.8,0.5){\line(1,0){1.85}}
\put(6.2,0.5){\makebox(0,0)[c]{$\bigotimes$}}
\put(6.6,0.5){\line(1,0){1.85}}
\put(9.25,0.5){\makebox(0,0)[c]{{\ovalBox(1.6,1.2){$\mc{L}_p$}}}}
\put(10.05,0.5){\line(1,0){1.85}}
\put(12.35,0.5){\makebox(0,0)[c]{$\bigcirc$}}
\put(12.8,0.5){\line(1,0){1.85}}
\put(15.45,0.5){\makebox(0,0)[c]{{\ovalBox(1.6,1.2){$\mc{K}$}}}}
\put(16.25,0.5){\line(1,0){1.85}}
\put(18.55,0.5){\makebox(0,0)[c]{$\bigotimes$}}
\put(19.0,0.5){\line(1,0){1.85}}
\put(21.65,0.5){\makebox(0,0)[c]{{\ovalBox(1.6,1.2){$\ov{\mc{R}}_q$}}}}
\put(22.45,0.5){\line(1,0){1.85}}
\put(24.75,0.5){\makebox(0,0)[c]{$\bigotimes$}}
\put(25.2,0.5){\line(1,0){1.85}}
\put(27.85,0.5){\makebox(0,0)[c]{{\ovalBox(1.6,1.2){$\mc{R}$}}}}
\put(6.5,-0.5){\makebox(0,0)[c]{\tiny$\epsilon_{-\hf}-\epsilon_{1-p}$}}
\put(12.5,-0.5){\makebox(0,0)[c]{\tiny$\epsilon_{-1}-\epsilon_{\bar{1}}$}}
\put(18.5,-0.5){\makebox(0,0)[c]{\tiny$\epsilon_{\ov{k}}-\epsilon_{\hf}$}}
\put(24.7,-0.5){\makebox(0,0)[c]{\tiny$\epsilon_{q-\hf}-\epsilon_{1}$}}
\end{picture}
\end{center}
\bigskip
\begin{center}
\hskip -3cm \setlength{\unitlength}{0.16in}
\begin{picture}(24,1)
\put(-.5,0.5){\makebox(0,0)[c]{$\ov{\mc{G}}^\diamond$:}}
\put(3.0,0.5){\makebox(0,0)[c]{{\ovalBox(1.6,1.2){$\mc{L}^\diamond$}}}}
\put(3.8,0.5){\line(1,0){1.85}}
\put(6.2,0.5){\makebox(0,0)[c]{$\bigotimes$}}
\put(6.6,0.5){\line(1,0){1.85}}
\put(9.25,0.5){\makebox(0,0)[c]{{\ovalBox(1.6,1.2){$\mc{L}_p$}}}}
\put(10.05,0.5){\line(1,0){1.85}}
\put(12.35,0.5){\makebox(0,0)[c]{$\bigcirc$}}
\put(12.8,0.5){\line(1,0){1.85}}
\put(15.45,0.5){\makebox(0,0)[c]{{\ovalBox(1.6,1.2){$\mc{K}$}}}}
\put(16.25,0.5){\line(1,0){1.85}}
\put(18.55,0.5){\makebox(0,0)[c]{$\bigcirc$}}
\put(19.0,0.5){\line(1,0){1.85}}
\put(21.65,0.5){\makebox(0,0)[c]{{\ovalBox(1.6,1.2){$\mc{R}_q$}}}}
\put(22.45,0.5){\line(1,0){1.85}}
\put(24.75,0.5){\makebox(0,0)[c]{$\bigotimes$}}
\put(25.2,0.5){\line(1,0){1.85}}
\put(27.85,0.5){\makebox(0,0)[c]{{\ovalBox(1.6,1.2){$\ov{\mc{R}}$}}}}
\put(6.5,-0.5){\makebox(0,0)[c]{\tiny$\epsilon_{-\hf}-\epsilon_{1-p}$}}
\put(12.5,-0.5){\makebox(0,0)[c]{\tiny$\epsilon_{-1}-\epsilon_{\bar{1}}$}}
\put(18.5,-0.5){\makebox(0,0)[c]{\tiny$\epsilon_{\ov{k}}-\epsilon_{1}$}}
\put(24.7,-0.5){\makebox(0,0)[c]{\tiny$\epsilon_{q}-\epsilon_{\hf}$}}
\end{picture}
\end{center}
\bigskip

Let $a,\la_{1},\ldots,\la_{k}\in\C$, $\la^-$ and $\la^+$ be two
partitions. Associated to the tuple
$\la=(a,\la_{1},\ldots,\la_{k};\la^-,\la^+)$ satisfying the dominant
condition \eqnref{dominant}, set
 {\allowdisplaybreaks
\begin{align*}
\La_F^+(\la) &:=\sum_{i=1}^{k}\la_{i}\epsilon_{\ov{i}}
+\sum_{j=1}^{q}(\la^+)'_{j}\epsilon_{j-\hf}+\sum_{j\in\N}\langle\la^+_{j}-q\rangle\epsilon_{j}\\
\ov{\La}_F^+(\la)&:=\sum_{i=1}^{k}\la_{i}\epsilon_{\ov{i}}
+\sum_{j=1}^q\la^+_{j}\epsilon_{j} +\sum_{j\in\N}\langle(\la^+)'_{j}-q\rangle\epsilon_{j-\hf},\\
{\La}_F^-(\la)&:=-\sum_{j=1}^p(\la^-)'_{j}\epsilon_{-j+\hf}
-\sum_{j\in\N}\langle\la^-_{j}-p\rangle\epsilon_{-j+1}+ a\La_0,\\
\ov{\La}_F^-(\la)&:=-\sum_{j=1}^p\la^-_{j}\epsilon_{-j+1}
-\sum_{j\in\N}\langle(\la^-)'_{j}-p\rangle\epsilon_{-j+\hf}+ a\La_0.
\end{align*}
 }
Associated to such tuple, we define the weights
\begin{align*}
{\la}_F&:={\La}^-(\la)+\La^+(\la),\quad\quad\quad
\ov{\la}_F:=\La^-(\la)+\ov{\La}^+(\la),\\
{\la}_F^\diamond&:=\ov{\La}^-(\la)+\La^+(\la),\quad\quad\qquad
\ov{\la}_F^\diamond:=\ov{\La}^-(\la)+\ov{\La}^+(\la).
\end{align*}
The sets $P^+_a$, $\ov{P}^+_a$, $P^{\diamond+}_a$ and
$\ov{P}^{\diamond+}_a$ of all dominant weights are the sets
consisting of the weights of the form $\la_F$, $\ov{\la}_F$,
$\la_F^{\diamond}$, $\ov{\la}_F^{\diamond}$, respectively. The main
results on super duality in Sections~\ref{sec:borel}, \ref{sec:O}
and \ref{sec:SDtilting} afford straightforward counterparts in this
new setting.

\begin{rem}
The $\mf u$-(co)homology formulas in the sense of Kostant and
Enright \cite{E} for unitarizable modules of general linear Lie
superalgebras have been computed in \cite{LZ} (cf. \cite[Theorem
5.1, Theorem 4.4]{HLT}), and they can be easily recovered via the
super duality of this paper by an appropriate choice of $\la$ and of
the orderings defined in the example above. Let us set $k=p=0$ and
choose $\la$ satisfying $\la^-_1+\la^+_1\le a$ with $a\in \Z_+$.
Then we have Enright type $\mf u$-(co)homology formulas for
unitarizable modules $L(\ov{\mc{G}}^\diamond,\ov{\la}^\diamond_F)$
(cf. \cite[Theorem 4.4]{HLT}). Using \thmref{matching:KL}, its
analogue for the irreducible modules
$L({\mc{G}}^\diamond,{\la}_F^\diamond)$ together with Enright's
(co)homology formulas for
$L(\ov{\mc{G}}^\diamond,\ov{\la}^\diamond_F)$, we recover
\cite[Theorem 6.1]{LZ}.
\end{rem}

 \vspace{.3cm}
{\bf Acknowledgments.} The first author is partially supported by an
NSC-grant and an Academia Sinica Investigator grant. The second
author is partially supported by an NSC-grant. The third author is
partially supported by an NSF grant.


\bigskip
\frenchspacing

\begin{thebibliography}{aABC}
\bibitem[BB]{BB} A. Beilinson and J. Bernstein, {\em Localisation de
$\mathfrak g$-modules}, C.R. Acad. Sci. Paris Ser. I Math. {\bf
292} (1981), 15--18.

\bibitem[BK]{BK} J.L.~Brylinski and M.~Kashiwara,
{\em Kazhdan-Lusztig conjecture and holonomic systerms}, Invent.
Math. {\bf 64} (1981), 387--410.

\bi[Br1]{Br1} J.~Brundan, {\em Kazhdan-Lusztig polynomials and
character formulae for the Lie superalgebra $\gl(m|n)$},
J.~Amer.~Math.~Soc.~{\bf 16} (2003), 185--231.

\bibitem[Br2]{Br2} J.~Brundan,
{\em Tilting modules for Lie superalgebras}, Commun.~Algebra~{\bf
32} (2004), 2251-2268.

\bi[BS]{BS} J.~Brundan and C.~Stroppel, {\em Highest weight
categories arising from Khovanov's diagram algebra IV: the general
linear supergroup}, preprint 2009, arXiv:math.RT/0907.2543.


\bi[CK]{CK} S.-J.~Cheng and J.-H.~Kwon,  {\em Howe duality and
Kostant's homology formula for infinite-dimensional Lie
superalgebras}, Int.~Math.~Res.~Not. {\bf 2008}, Art.~ID rnn 085,
52 pp.

\bibitem[CKW]{CKW} S.-J.~Cheng, J.-H.~Kwon, and W. Wang,
{\em Kostant homology formulas for oscillator modules of Lie
superalgebras}, Adv.~Math.~{\bf 224} (2010), 1548--1588.


\bi[CL1]{CL0} S.-J. Cheng and N.~Lam,
{\em Infinite-dimensional Lie
superalgebras and hook Schur functions}, Commun. Math. Phys. {\bf
238} (2003), 95--118.

\bi[CL2]{CL} S.-J. Cheng and N.~Lam, {\em Irreducible characters of
general linear superalgebra and super duality},
Commun.~Math.~Phys.~{\bf 280} (2010), 645--672.

\bibitem[CLW]{CLW} S.-J. Cheng, N.~Lam and W.~Wang,
{\em Super duality and irreducible characters of ortho-symplectic
Lie superalgebras}, Invent.~Math.~{\bf 183} (2011), 189--224.

\bi[CW]{CW} S.-J.~Cheng and W.~Wang, {\em Brundan-Kazhdan-Lusztig
and Super Duality Conjectures}, Publ. Res. Inst. Math. Sci. {\bf
44} (2008), 1219--1272.

\bi[CWZ]{CWZ} S.-J.~Cheng, W.~Wang and R.B.~Zhang, {\it Super
duality and Kazhdan-Lusztig polynomials}, Trans.~Amer.~Math.~Soc.~
{\bf 360} (2008), 5883--5924.


\bi[Deo]{Deo} V. Deodhar, {\em On some geometric aspects of Bruhat
orderings II: the parabilic analogue of Kazhdan-Lusztig
polynomials}, J.~Algebra {\bf 111} (1987) 483--506.

\bibitem[Don]{Don} S. Donkin,
{\em On tilting modules for algebraic groups}, Math. Z. {\bf 212}
(1993), 39--60.

\bibitem[E]{E} T.J.~Enright,
{\em Analogues of Kostant's $\mathfrak{u}$-cohomology formula for
unitary highest weight modules}, J. Reine Angew. Math. {\bf 392}
(1988) 27-36.



\bibitem[HLT]{HLT}P.-Y.~Huang, N.~Lam and T.-M.~To,
{\em Super duality and homology of unitarizable modules of Lie
algebras}, Publ. Res. Inst. Math. Sci. to appear, arXiv:1012.1087.



\bibitem[KMS]{KMS}
M.~Kashiwara, T.~Miwa, and E.~Stern, {\em Decomposition of
$q$-deformed Fock spaces}, Selecta Math. (N.S.) {\bf 1} (1995),
787--805.

\bibitem[KR]{KR} V.~Kac and A.~Radul,
{\em Representation theory of the vertex algebra $W_{1+\infty}$},
Transform. Groups {\bf 1} (1996), 41--70.

\bibitem[KW]{KW} V.~Kac and M.~Wakimoto,
 {\em Integrable highest weight modules over affine superalgebras and Appell's function},
 Commun. Math. Phys. {\bf 215} (2001), 631--682.

\bibitem[KS]{KS} M.~Kashiwara and P.~Schapira, {\em Categories and sheaves},
Grundlehren der Mathematischen Wissenschaften, {\bf 332}, Springer-Verlag, Berlin, 2006.

\bibitem[KL]{KL} D. Kazhdan and G. Lusztig,
{\em Representations of Coxeter groups and Hecke algebras},
Invent. Math. {\bf 53} (1979), 165--184.

\bibitem[LZ]{LZ} N.~Lam and R.B.~Zhang,
{\em $u$-cohomology formula for unitarizable modules over general
linear superalgebras}, J.~Algebra {\bf 327} (2011), 50--70.

\bibitem[Lu]{Lu} G.~Lusztig,
{\em Introduction to quantum groups}, Progress in Math. {\bf 110},
Birkh\"auser, 1993.




\bi[PS]{PS} I.~Penkov and V.~Serganova, {\em Cohomology of $G/P$ for
classical complex Lie supergroups $G$ and characters of some
atypical $G$-modules}, Ann. Inst. Fourier {\bf 39} (1989), 845--873.





\bibitem[So]{So}  W. Soergel,
{\em Character formulas for tilting modules over Kac-Moody
algebras}, Represent. Theory (electronic) {\bf 2} (1998),
432--448.

\bi[Ta]{T} J.~Tanaka, {On homology and cohomology of Lie
superalgebras with coefficients in their finite-dimensional
representations}, Proc.~Japan Acad.~Ser.~A Math.~Sci.~{\bf 71}
(1995), no. 3, 51--53.

\bi[Vo]{V} D.~Vogan, {\em Irreducible characters of semisimple Lie
Groups II: The Kazhdan-Lusztig Conjectures}, Duke Math.~J.~{\bf
46} (1979), 805--859.
\end{thebibliography}
\end{document}